\documentclass[12pt,twoside]{article}

\usepackage{graphicx,graphics, amsthm, amsfonts, amsbsy,amssymb, amsmath, cite, array,arydshln, hyperref, float,tikz,standalone}
\usepackage[utf8]{inputenc}
\usetikzlibrary{shapes.geometric,arrows.meta,positioning}

\usepackage[T1]{fontenc}
\usepackage[utf8]{inputenc}

\everymath{\displaystyle}

\let\OLDthebibliography\thebibliography
\renewcommand\thebibliography[1]{
	\OLDthebibliography{#1}
	\setlength{\parskip}{3pt}
	\setlength{\itemsep}{0pt plus 0.3ex}
}

\usepackage[margin=1in]{geometry}
\usepackage[small,compact]{titlesec}

\allowdisplaybreaks

\newtheorem{theorem}{Theorem}[section]

\newtheorem{problem}{Problem}

\newtheorem{proposition}[theorem]{Proposition}

\newtheorem{conjecture}{Conjecture}

\renewenvironment{proof}{{\noindent\bfseries Proof.}}{\qed}

\providecommand{\keywords}[1]
{
	\noindent\textbf{Keywords:} #1
}
\providecommand{\ams}[2]
{
	\noindent\textbf{2020 AMS subject classification:} #2
}

\def\qed{\nolinebreak\hfill\rule{.2cm}{.2cm}\par\addvspace{.5cm}}

\begin{document}
\title{Combinatorial and analytic aspects of independence polynomials of zero divisor graphs}
\author{Bilal Ahmad Rather\\
	{\em\small School of Mathematics and Statistics, Shandong University of Technology,}\\
	{\em \small Zibo 255049, China}\\
	\texttt{bilalahmadrr@gmail.com}
				}
				
\date{}

\pagestyle{myheadings} \markboth{Bilal Ahmad Rather}{Combinatorial and analytic aspects of independence polynomials of zero divisor graphs}
\maketitle

\begin{abstract}
The independence polynomial of a graph encapsulates all independent sets of differing sizes, a task classified as NP-hard in theoretical computer science. This article examines the independence polynomial of zero divisor graphs in commutative rings. We demonstrate that the independent sets, represented as a sequence of coefficients of the independence polynomial, exhibit unimodality and log-concavity. Therefore, for the independence polynomial of some zero divisor graphs, the unimodal conjecture is true. Additionally, the characteristics of the zeros of the independence polynomial are delineated, along with their corresponding annular regions on the plane. 
\end{abstract} 
 

\ams{}{05C25, 05C31, 68R05, 05C69, 26C10, 68Q25.}
\section{Introduction}
All graph sconsidered in this article are finite, simple, connected and undirected.  A graph is denoted as $G=G(V(G),E(G))$, where $V(G)$ is the vertex set, and $E(G)$ is the edge set. The cardinal number $|E(G)|$ is the size $m,$ and $|V(G)| $ is the order $n$ of $G$. The number of edges incident on a vertex $ v_{i}\in V(G) $ is the \textit{degree} of $v_{i}$, and is denoted as $ d_{v_{i}}(G) $.  The \textit{union} of two graphs $ G_{1}$ and $ G_{2}$ with disjoint vertex sets $V_{1}, $ and $ V_{2}$, and disjoint edge sets $E_{1}$ and $E_{2}$,  denoted by $ G_{1}\cup G_{2}$,  is defined as a graph with vertex set $ V_{1}\cup V_{2} $ and edge set $ E_{1}\cup E_{2}. $ 
The join of two graphs $ G_{1}$ and $ G_{2}$, denoted by $G_{1}\vee G_{2}$, is a graph with vertex set $V_{1}\cup V_{2}$, and edge set $E_{1}\cup E_{2}, $ together with edges $\big\{\{u,v\} : u\in V(G_{1}), v\in V(G_{2}) \big\}.$\vskip 1mm

A polynomial $ p(x)=\sum_{i=0}^{n} \ell_{i}x^{i} $ is said to be \textit{unimodal}, if the coefficient sequence $\Big\{\ell_{i}\Big\}_{i=0}^{n}$ form a unimodal sequence, that is, we can find a positive integer $ p~(0\leq p\leq n), $ such that $ \ell_{0}\leq \ell_{1}\leq \dots \leq \ell_{p} \geq \ell_{p+1}\geq \dots \geq \ell_{n} $. So, the coefficients of $p(x)$ increase to a certain value, and thereafter  decreases. The polynomial $ p(x) $ is said to be \emph{log-concave} if 
\begin{equation}\label{log con 1}
	\ell_{j}^{2}\geq \ell_{j-1}\ell_{j+1},\quad \text{for all}\quad  1\leq j\leq n-1. 
\end{equation} 

A positive log-concave sequence is known to be unimodal \cite{stanley}. However, if $\ell_{i}=0$ implies that either $\ell_{1}=\dots=\ell_{i-1}=0$ or $\ell_{i+1}=\dots=\ell_{n}=0,$ (no internal zeros), then a non-negative log-concave sequence is unimodal \cite{stanley}. In topological graph theory, log-concavity of sequences is connected with surface embedding of graphs, comparable to log-concavity of genus polynomials of graphs \cite{gross}. The polynomial $ 1 + 2x + 3x^{2} + 4x^{3} + 3x^{4} + 2x^{5}+x^{6} $ is unimodal, while $ 12 + 7x + 9x^{2} + 16x^{3} + 4x^{4} $ is not. A log-concave polynomial is $p(x) = x^5 + 5x^4 + 10x^3 + 10x^2 + 5x + 1,$ while $q(x) = x^5 + 10x^4 + 5x^3 + 5x^2 + 10x + 1$ is not.
Many polynomials related to graph invariants like matching \cite{heilmann}, chromatic \cite{huh}, and reliability polynomials \cite{huh1}, are thought to exhibit unimodal coefficient (absolute value) sequences under different base expansions. However, there are still conjectures open related to independence polynomial of graphs \cite{alavi,levitejc,micheal,browncameron}.
 These polynomial characteristics are mostly investigated for graph polynomials, such as the independent domination polynomial \cite{bilaltcs, bilaljcmcc}, dominating polynomial, matching polynomial, clique polynomial, and several others \cite{huh1,gross, bilaldml,bilalsc,bilaltcs,bilaljac}. 

\medskip

If $ R $ is a commutative ring $1\neq 0$, a non-zero element $ x\in R $ is a zero divisor of $ R $ if there is a non-zero $ y\in R$ such that $ x\cdot y=0. $ Beck \cite{ib} introduced the concept of zero divisor graphs to study the coloring of rings. Beck defined zero divisor graph $\varGamma(R)$, with vertices as zero divisors including identity $0$, and connected two distinct vertices if and only if their product is zero. So, it follows that $0$ is adjacent to all vertices of $\varGamma(R)$. 
Anderson and Livingston \cite{al}, omitting $ 0 $ in $\varGamma(R)$, connected two non-zero zero divisors if and only if their product is zero.
A vast literature is available on zero divisor graphs and their applications.  Anderson and Weber \cite{anderson2019} characterized all zero divisor graphs with at most $14$ vertices.  In \cite{alikhanizd}, the authors established the domination polynomial of the zero divisor graph of rings. Gürsoy et al.  \cite{gursoy} introduced the idea of independent domination polynomials for zero divisor graphs in $\mathbb{Z}_{n}$ for specific values of $n.$ The theory of independent domination polynomial theory is recently developed in \cite{bilaltcs}. The zero divisor graphs are universal graphs, that is, for an integer modulo ring $\mathbb{Z}_{n}$ (or even in boolean ring), given any finite graph of order $n^{\prime}$, there is some $n$ such that $G$ is an induced subgraph of zero divisor graph $ \varGamma(\mathbb{Z}_{n})$ \cite{arunkumara,Andersonbook}. Some interesting properties of zero divisor graphs were studied in \cite{al,asir2}. We will revise basic properties of zero divisor graphs of commutative ring $\mathbb{Z}_{n}.$

We will discuss the structure of $\varGamma(\mathbb{Z}_{n})$. If  $ \rho_{i}$ be the proper divisor of $n$, then we consider the sets
$$ V_{\rho_{i}}= \{ x\in \mathbb{Z}_{n} : \gcd(x,n)=\rho_{i} \}, \quad \text{for}\quad  1\leq i \leq k, $$
where $k$ is some positive integer other than $1$ and $n.$  Also, note that $ \gcd(x,n) $ is the greatest common divisor of $ x $ and $ n. $
Also, $ V_{\rho_{i}} \cap V_{\rho_{j}}=\emptyset$ for $ i\neq j $. Thus, the vertex set of  $\varGamma(\mathbb{Z}_{n})$  can be partitioned as $ V_{\rho_{1}}, V_{\rho_{2}}, \dots, V_{\rho_{k}}$. Also,  in graph $\varGamma(\mathbb{Z}_{n})$, the vertices of $ V_{\rho_{i}} $ are adjacent to vertices of $ V_{\rho_{j}}  $  if and only if $ \rho_{i}\rho_{j}$ is multiple of $n.$ The cardinality of $V_{\rho_{i}} $ is $\varphi\left( \tfrac{n}{\rho_{i}}\right)$, for $ 1\leq i \leq k $, where $\varphi$ is the Euler's Totient function.  In addition, the induced subgraph of $ V_{\rho_{i}} $ is either a clique or its complement.
Furthermore, the induced subgraph of $ \varGamma(V_{\rho_i}) $ is $ K_{\varphi\left (\frac{n}{\rho_{i}}\right )} $ if $ \rho_{i}^{2} $ divides $ n, $ otherwise it is $\overline{K}_{\varphi\left (\frac{n}{\rho_{i}}\right )}$. A complete graph is denoted by $K_{n},$ and its complement by $\overline{K}_{n}.$

The article is organized as: Section \ref{section 2} gives the independence polynomial of $\varGamma(\mathbb{Z}_{n})$ for $n\in \{p,p^{2},p^{3},pq\}$, where $p<q$ are primes. Also, we find their unimodal, log-concave properties and discuss their zeros. In Section \ref{section 3}, we give the independence polynomial of $\varGamma(\mathbb{Z}_{p^{2}q})$, investigate its zero, log-concave and unimodal properties. In Section \ref{section 4}, we present the independence polynomial of $\varGamma(\mathbb{Z}_{pqr})$, where $p<q<r$ are primes. We give bounds for its zeros and illustrate with examples.

\section{Independence polynomial of zero divisor graph of integral modulo ring}\label{section 2}
An \emph{independent set} in a graph $G$ is a set of vertices no two of which are adjacent. For a graph $G,$ the cardinality of the maximum size of an independent set, is called the \emph{independence number} of $G$, and is denoted by $\alpha(G)$. The associated \emph{decision problem}
$$
\text{Independent set}=\{(G,k): \alpha(G)\ge k\},
$$
is one of the classical NP-complete problems, and the optimization problem of computing $\alpha(G)$ is therefore NP-hard, where $k$ is a positive integer. Karp showed in his seminal list of NP-complete problems that the independent set problem is NP-complete, and a fundamental problem \cite{karp1972,garey1979}.  Since, the independent set problem is equivalent (under complementation) to clique set in $G$, known inapproximability bounds for clique number transfer immediately. It is NP-hard to approximate the maximum independent set within a factor of $n^{1-\varepsilon}$ for any fixed $\varepsilon>0$, unless $\mathrm{P}=\mathrm{NP}$ \cite{Hastad99}.   The NP-completeness of independent set implies that, unless $\mathrm{P}=\mathrm{NP}$, there is no polynomial-time algorithm that computes $\alpha(G)$ on general graphs. Moreover, the strong inapproximability results rule out polynomial-time algorithms with approximation ratio $n^{1-\varepsilon}, \varepsilon>0$, placing the problem among the most intractable maximization problems from the perspective of worst-case approximation.

Let $s_k(G)$ be the number of independent sets of size $k$ in $G$. Then the \emph{independence polynomial} of $G$ is the generating function of possible independent sets, and is given as
$$
I(G,x)\;=\;\sum_{k=0}^{\alpha(G)} s_k(G)\,x^k.
$$
The polynomial $I(G,x)$ was introduced and studied systematically in the 1980s and has since become a standard enumerative invariant with connections to statistical physics (the hard–core model), combinatorial optimization, and algebraic graph theory \cite{gutman,levitsurvey}. Evaluating $I(G,x)$ at special points recovers several natural quantities. In particular, $I(G,1)=\sum_{k\ge 0}s_k(G)$ counts all independent sets of $G$, while $I(G,0)=1$ and $I'(G,0)=s_1(G)=|V(G)|$. For $x>0$, $I(G,x)$ is (up to a trivial change of variable) the partition function of the hardcore lattice gas on $G$ with activity $x$; this perspective underlies many structural and analytic results \cite{Weitz}.
If $G$ and $H$ are vertex disjoint graphs, then
$$
I(G\cup H,x)\;=\;I(G,x)\,I(H,x),
$$
where $G\cup H$ is the union of graphs. The above expression follows as an independent set in $G\cup H$ is precisely the disjoint union of an independent set of $G$ and one of $H$ \cite{gutman,levitsurvey, levitdiscrete,haynes}. Let $L(G)$ be the line graph of $G$ and $m_k(G)$ be the number of $k$-matchings of $G$, then we have
$$
\sum_{k\ge 0} m_k(G)\,x^k\;=\;I\bigl(L(G),x\bigr).
$$
Thus, many properties of independence polynomials transfer to matching generating functions, and vice versa \cite{levitsurvey}.
The coefficient sequence $\bigl(s_0(G),s_1(G),\dots,s_{\alpha(G)}(G)\bigr)$ is strictly positive and obeys basic combinatorial inequalities. In particular, when $I(G,x)$ has only real zeros, then the Newton's inequalities imply that the coefficient sequence is log-concave, and hence unimodal \cite{levitsurvey}. If $G$ is claw-free (that is, contains no induced star $K_{1,3}$), then $I(G,x)$ has only real zeros \cite{ChudnovskySeymour2007}. Consequently its coefficients are log-concave and unimodal. This result is sharp, in general graphs, independence polynomials need not be real-rooted. Active topics include root location for broader graph classes, stability and zero-free regions motivated by statistical physics, extremal problems for the coefficient sequences, and algorithmic approximation on sparse graph families. Surveys such as \cite{levitsurvey,goddard} provide a comprehensive entry point and extensive bibliographies. A graph of order $n$ is very well-covered if every maximal independent set has size  $\tfrac{n}{2}$, like the regular complete bipartite graphs $K_{m,m}$ is very well-covered.
A graph is well-covered graphs, if all its maximal independent sets have the same size (complete graphs and the cycle are examples).
In \cite{brown}, the authors conjectured that the independence coefficients of well-covered graphs were unimodal, and showed that every graph 
can be embedded as an induced subgraph of such a well-covered graph. However, Michael and Traves \cite{micheal} later disproved the conjecture. A conjecture due to Alavi et al. \cite{alavi} that is still open is that the independence polynomial of a tree is unimodal. Levit and Mandrescu \cite{levitejc} conjectured that the independence polynomial of every very well-covered graphs are unimodal, and to date, the conjecture remains open. Though there are progress for certain families of graphs \cite{micheal,browncameron}. 
The independence polynomial of $ G_{1}\vee G_{2} $ ((see, \cite{gutman}) is 
\begin{equation}\label{eq join}
	I(G_{1}\vee G_{2},x)=I(G_{1},x)+I(G_{2},x)-1.
\end{equation} With this information and the fact that the independence polynomial of $\overline{K}_{n}$ is $(1+x)^{n}$, and that of $K_{n}$ is $1+nx$. Next, we discuss the independence polynomial of zero divisor graphs of ring $\mathbb{Z}_{n}$ for certain values of $n.$
\vskip 2mm

For  $n=p$, the ring is $\mathbb{Z}_p$, which is a field, since $p$ is prime. A field has no zero divisors other than 0. The set of vertices of the zero divisor graph $\varGamma(\mathbb{Z}_p)$ is the set of non-zero zero divisors, which is an empty graph. 
The independence polynomial is $I(\varGamma(\mathbb{Z}_p), x) = i_0 x^0 = 1$.

For  $n=p^2$, the set of non-zero zero divisors in $\mathbb{Z}_{p^2}$ are the elements $a \in \{1, \dots, p^2-1\}$ such that $\gcd(a, p^2) > 1$. This is equivalent to $p\mid a$, where $\mid$ means divide. The vertex set is $V = \{kp \mid 1 \le k \le p-1\}$, and  $|V|=p-1$. Let $v_1 = k_1 p$ and $v_2 = k_2 p$ be two distinct vertices in $V$, where $1 \le k_1, k_2 \le p-1$ and $k_1 \neq k_2$. Their product is $v_1 v_2 = (k_1 p)(k_2 p) = k_1 k_2 p^2 \equiv 0 \pmod{p^2}$.
Thus, every pair of distinct vertices is adjacent. The graph $\varGamma(\mathbb{Z}_{p^2})$ is the complete graph $K_{p-1}$.
The independence polynomial of $\varGamma(\mathbb{Z}_{p^2}) $ is:
$$I(\varGamma(\mathbb{Z}_{p^2}), x) = i_0 x^0 + i_1 x^1 = 1 + (p-1)x.$$

For $n=p^{3}$, we have the following proposition.
\begin{proposition}
	Let $p$ be a prime and $\varGamma(\mathbb{Z}_{p^3})$ be zero divisor graph of $\mathbb{Z}_{p^3}$. Then, the independence polynomial of $\varGamma(\mathbb{Z}_{p^3})$ is
	$$I(\varGamma(\mathbb{Z}_{p^3}), x) = (1+x)^{p^2-p} + (1+(p-1)x) - 1 = (1+x)^{p^2-p} + (p-1)x.$$
\end{proposition}
\begin{proof}
	For $n=p^3$, the set of non-zero zero divisors in $\mathbb{Z}_{p^3}$ are the multiples of $p$ in $\{1, \dots, p^3-1\}$. We partition the vertex set $V$ into two disjoint subsets: $V_1 = \{a \in V \mid \gcd(a, p^3)=p\} = \{kp \mid 1 \le k < p^2, p \nmid k\}$, and  $V_2 = \{a \in V \mid \gcd(a, p^3)=p^2\} = \{kp^2 \mid 1 \le k < p\}$, where $\nmid$ represents does not divide. The sizes of these sets are:  $|V_1| = (p^2-1) - (p-1) = p^2-p$ and  $|V_2| = p-1$.  Let $a=k_1 p$ and $b=k_2 p$ be distinct vertices in $V_1$, where $p \nmid k_1$ and $p \nmid k_2$. Then their product is $ab = k_1 k_2 p^2$. Since $p$ is prime, it does not divide $k_1$ or $k_2$, and $p \nmid k_1 k_2$. Thus, $ab \not\equiv 0 \pmod{p^3}$, and it follows that there are no edges between any two vertices in $V_1$. The subgraph induced by $V_1$ is an edgeless graph $\overline{K}_{p^2-p}$. Let $a=k_1 p^2$ and $b=k_2 p^2$ be distinct vertices in $V_2$. Then their product is $ab = k_1 k_2 p^4 = (k_1 k_2 p) p^3 \equiv 0 \pmod{p^3}$. Thus, every pair of distinct vertices in $V_2$ is adjacent. So, the subgraph induced by $V_2$ is a complete graph $K_{p-1}$. Let $a=k_1 p \in V_1$ and $b=k_2 p^2 \in V_2$. Their product is $ab = (k_1 p)(k_2 p^2) = k_1 k_2 p^3 \equiv 0 \pmod{p^3}$. So, if follows that every vertex in $V_1$ is adjacent to every vertex in $V_2$. Thus, $ \varGamma(\mathbb{Z}_{p^3})\cong K_{|V_{2}|}\vee \overline{K}_{|V_{1}|}.$ Keeping in mind the independence polynomial of $\overline{K}_{n}$ and $K_{n}$, with \eqref{eq join}, we have
	$$I(\varGamma(\mathbb{Z}_{p^3}), x) = (1+x)^{p^2-p} + (1+(p-1)x) - 1 = (1+x)^{p^2-p} + (p-1)x.$$
\end{proof}

The following result gives unimodal and log-concave property of $I(\varGamma(\mathbb{Z}_{p^3}), x).$
\begin{theorem}
	The independence polynomial $I(\varGamma(\mathbb{Z}_{p^3}), x)$ of the zero divisor graph $ \varGamma(\mathbb{Z}_{p^3})$, with prime $p$ is log-concave and  unimodal.
\end{theorem}
\begin{proof}
	Let $P(x) = I(\varGamma(\mathbb{Z}_{p^3}), x)$, and with $n = p^2-p$, we have
	$$P(x) = (1+x)^n + (p-1)x = \sum_{k=0}^n \binom{n}{k}x^k + (p-1)x.$$
	Let the coefficient sequence of $P(x)$ be $(\ell_k)_{k=0}^n$. Then we have
	$\ell_0 = \binom{n}{0} = 1, \ell_1 = \binom{n}{1} + p-1 = (p^2-p) + p-1 = p^2-1$, and  $\ell_k = \binom{n}{k}$ for $k \in \{2, 3, \dots, n\}.$
	Clearly, the coefficients of $P(x)$ are positive for $p \ge 2$. For $k=1$, we must have $\ell_1^2 \ge \ell_0 \ell_2$. Which translates to $(p^2-1)^2 \ge 1 \cdot \binom{p^2-p}{2}$, and gives $(p-1)^2(p+1)^2 \ge \tfrac{p(p-1)(p^2-p-1)}{2}$. It simplifies to  $2(p-1)(p+1)^2 \ge p(p^2-p-1)$, which is true for $p\geq 2.$ For $k=2$,  $\ell_2^2 \ge \ell_1 \ell_3$ implies that  $\binom{p^2-p}{2}^2 \ge (p^2-1)\binom{p^2-p}{3}$.
	For $p=2$,  the inequality $\ell_2^2 \ge 0$ is trivially true. For $p \ge 3$, with $n=p^2-p \ge 6$, the inequality is
	$\left(\tfrac{n(n-1)}{2}\right)^2 \ge (p^2-1)\tfrac{n(n-1)(n-2)}{6}$. And by dividing $\tfrac{n(n-1)}{2}$, it yields $\tfrac{n(n-1)}{2} \ge (p^2-1)\tfrac{n-2}{3}$.
	$3n(n-1) \ge 2(p^2-1)(n-2)$, which after simplification gives $(p-1)^3+5 \ge 0$. Thus, the inequality holds. For  $k \ge 3$, and consider  $\ell_{j} = \binom{n}{j}$ for $j \in \{k-1, k, k+1\}$. So, we must show $\binom{n}{k}^2 \ge \binom{n}{k-1}\binom{n}{k+1}$, which is a standard property of binomial coefficients. It is equivalent to the ratio test $\tfrac{\binom{n}{k}}{\binom{n}{k-1}} \ge \tfrac{\binom{n}{k+1}}{\binom{n}{k}}$, which is $\tfrac{n-k+1}{k} \ge \tfrac{n-k}{k+1}$. This simplifies to $(n-k+1)(k+1) \ge k(n-k)$, which reduces to $n+1 \ge 0$. That is true as $n=p^2-p \ge 2$. Thus, the coefficient sequence $(\ell_k)_{k=0}^n$ of $P(x)$ is log-concave, and a log-concave positive real sequence is  unimodal. 
\end{proof}

Next, we discuss the independent  polynomial of the zero divisor graph of $\mathbb{Z}_{pq}$ with primes $p<q$.

For $n=pq$ with primes $p<q$, and let $R = \mathbb{Z}_{pq}$. Then the vertices of the graph $\varGamma(R)$ are the non-zero zero divisors of $R$. An element $a \in \{1, \dots, pq-1\}$ is a zero divisor if and only if $\gcd(a, pq) \neq 1$, which is equivalent to $p|a$ or $q|a$. The vertex set of $\varGamma(\mathbb{Z}_{pq})$ is partitioned into two sets: $V_p = \{kp \mid 1 \le k \le q-1\}$ and $V_q = \{lq \mid 1 \le l \le p-1\}$. The zero divisor $\varGamma(\mathbb{Z}_{pq})$ is the complete bipartite graph $K_{q-1, p-1}$ with parts $V_p$ and $V_q$.  Let $x,y \in V_p$ be distinct vertices. Then $x=ap$ and $y=bp,$ for $1 \le a, b \le q-1$. Their product $xy = abp^2$ is zero in $\mathbb{Z}_{pq}$ if and only if $pq|abp^2$, which implies $q|abp$. As $\gcd(p,q)=1$, this requires $q|ab$. Since $q$ is prime, $q|a$ or $q|b$, which is impossible as $1 \le a,b \le q-1$. Thus, no two vertices in $V_p$ are adjacent. By a symmetric argument, no two vertices in $V_q$ are adjacent.  Let $x \in V_p$ and $y \in V_q$. Then $x=ap$ and $y=bq$ for some integers $a,b$. Their product is $xy = (ap)(bq) = ab(pq) \equiv 0 \pmod{pq}$. Thus, every vertex in $V_p$ is adjacent to every vertex in $V_q$.
So, it follows that $\varGamma(\mathbb{Z}_{pq})$ is a complete bipartite graph with partitions $V_p$ and $V_q$. 

Applying \eqref{eq join}, and noting that  $\varGamma(\mathbb{Z}_{pq}) \cong K_{q-1, p-1}\cong \overline{K}_{p-1}\vee \overline{K}_{q-1}$, we have the independence polynomial of $ \varGamma(\mathbb{Z}_{pq})$ as given below
\begin{equation}\label{ind poly zpq}
	I(\varGamma(\mathbb{Z}_{pq}), x) = (1+x)^{q-1} + (1+x)^{p-1} - 1.
\end{equation}

The following result gives the details about the real zeros of $I(\varGamma(\mathbb{Z}_{pq}), x).$
\begin{theorem}\rm\label{zeros pq}
	Let $x$ be a zero of the polynomial $I(\varGamma(\mathbb{Z}_{pq}), x)$. The properties of the zeros depend on the primes $p < q$.
	\begin{enumerate}
		\item If $p, q$ are odd primes. Then there are exactly two real zeros, one in the interval $(-1, 0)$ and another in $(-2, -1)$.
		\item If $p=2$ and $q$ is an odd prime, there are exactly two real zeros, one in the interval $(-1, 0)$ and another in $(-3, -2)$.
	\end{enumerate}
\end{theorem}
\begin{proof}
	Let $P(x) = (1+x)^{q-1} + (1+x)^{p-1} - 1$ be an independence polynomial of $\varGamma(\mathbb{Z}_{pq})$.  With $y=1+x$, the problem reduces to finding the roots of the equation $y^{q-1} + y^{p-1} = 1$. We analyze the roots based on the parity of the primes $p<q$. If $p $ and $q$ are odd primes, then $p-1=2k$ and $q-1=2m$ for some integers $k,m \ge 1$. Thus, the equation becomes $y^{2m} + y^{2k} = 1$. For real $y$, let $z=y^2 \ge 0$. Then the equation becomes $z^m + z^k = 1$. If $f(z) = z^m + z^k - 1$, then for $z>0$, we get $f'(z) = mz^{m-1} + kz^{k-1} > 0$. So, the function $f(z)$ is strictly increasing. Clearly, $f(0)=-1$ and $f(1)=1$, so by the Intermediate Value Theorem, there is a unique positive real root $z_0 \in (0,1)$. This gives two real roots for $y$, namely $y = \pm\sqrt{z_0}$. Since $0 < z_0 < 1$, we have $0 < \sqrt{z_0} < 1$. The corresponding zeros of $P(x)$ are $x_1 = \sqrt{z_0}-1 \in (-1,0)$ and $x_2 = -\sqrt{z_0}-1 \in (-2,-1)$. That proves (1). 
	
	If $p=2$, $q$ is an odd prime, then the equation for $y=1+x$ becomes $y^{q-1} + y = 1$.
	With $f(y) = y^{q-1}+y-1$, we have $f'(y) = (q-1)y^{q-2}+1$. As $q$ is an odd prime, $q-2$ is odd, and  $f'(y)=0$ has a unique real solution $y_c = -(1/(q-1))^{1/(q-2)} \in (-1,0)$, which is a local minimum. Also,
	$f(0)=-1$ and $f(1)=1$, so by Intermediate Value Theorem, there is a root $y_1 \in (0,1)$.
	As $y \to \pm\infty$, then $f(y) \to +\infty$. The minimum value is $f(y_c) = y_c(1 - 1/(q-1)) - 1 < -1$. Thus, there must be another real root $y_2 < y_c < 0$. Since $f(-1)=-1$ and $f(-2) = (-2)^{q-1}-3 = 2^{q-1}-3 \ge 1$ for $q \ge 3$, this second root lies in $(-2, -1)$.
	The corresponding zeros of $P(x)$ are $x_1=y_1-1 \in (-1,0)$ and $x_2=y_2-1 \in (-3,-2)$. That proves part (2). 
\end{proof}

For the remaining zeros of $I(\varGamma(\mathbb{Z}_{pq}), x),$ we have the following result.
\begin{theorem}\label{zeros pq complex}
	Let $x$ be a zero of the polynomial $P(x)=I(\varGamma(\mathbb{Z}_{pq}), x)$. With $p < q$, the zeros $x$ of $P(x)$ lie in the open annular region $r_1 < |x| < r_2$, where
	$$r_1 = \frac{R-1}{R^{q-1} + R^{p-1} - 2} \quad \text{and} \quad r_2 = R+1,\quad \text{with}\quad R = 2^{1/(q-p)}.$$
\end{theorem}
\begin{proof}
	By Equation \eqref{ind poly zpq}, the independence polynomial of $ \varGamma(\mathbb{Z}_{pq})$ is 
	$$P(x)=I(\varGamma(\mathbb{Z}_{pq}), x) = (1+x)^{q-1} + (1+x)^{p-1} - 1.$$
	Let $z$ be a zero of the polynomial $P(x)$, such that
	$$P(z) = (1+z)^{q-1} + (1+z)^{p-1} - 1 = 0.$$
	Let $w = 1+z$. Then $z=w-1$, and $w$ must be a zero of the new polynomial $Q(w) = w^{q-1} + w^{p-1} - 1$. The degree of $P(x)$ is $q-1$, so there are $q-1$ such zeros. Note that $P(0)=1$, so $z=0$ is not a zero, and consequently $w=1$ is not a zero of $Q(w)$. We first establish bounds for the modulus of the zeros of $Q(w)$. We claim that  any zero $w$ of $Q(w)=w^{q-1} + w^{p-1} - 1$ satisfies $|w| < 2^{1/(q-p)}$. 	The equation $Q(w)=0$ can be written as $w^{q-1} = 1 - w^{p-1}$. If $|w| \le 1$, the claim holds, since $p, q$ are primes with $p<q$, so $q-p \ge 1$, which implies $2^{1/(q-p)} > 1$. Now, if $|w|>1$, then with the help of the triangle inequality, we have
	$$|w|^{q-1} = |1 - w^{p-1}| \le 1 + |w|^{p-1}.$$
	If $x = |w|$, then with $x>1$ and $p-1>0$, we can divide by $x^{p-1}$, and obtain
	$$x^{q-p} \le \frac{1}{x^{p-1}} + 1 < 1+1 = 2.$$
	The above inequality is strict as $x>1$. Thus, $|w|^{q-p} < 2$, which implies that $|w| < 2^{1/(q-p)}$. So, with $R = 2^{1/(q-p)}$, it shows that all zeros $w$ of $Q(w)$ lie inside the circle $|w|=R$. Next, we find the annular region for the zeros $z=w-1$. The modulus of a zero $z$ is given by $|z|=|w-1|$. By the triangle inequality and by $|w|^{q-p} < 2$, we have
	$$|z| = |w-1| \le |w|+1 < R+1 = 2^{1/(q-p)}+1.$$
	Next, as $w$ is a zero of $Q(w)$, so $Q(w)=0$. Also $w \ne 1$, and $Q(1) = 1^{q-1}+1^{p-1}-1 = 1$. Now consider the difference $Q(w)-Q(1)=(w^{q-1}+w^{p-1}-1) - 1 = -1.$
	Also, we have
	$$Q(w)-Q(1) = (w^{q-1}-1) + (w^{p-1}-1)$$
	Comparing and factoring the differences of powers, we have
	$$(w-1)\left(\sum_{k=0}^{q-2} w^k\right) + (w-1)\left(\sum_{k=0}^{p-2} w^k\right) = -1.$$
	With $z=w-1$, and $S(w) = \sum_{k=0}^{q-2} w^k + \sum_{k=0}^{p-2} w^k$, we get $z \cdot S(w) = -1$.
	Since $z \ne 0$, so $S(w) \ne 0$, and we can write $|z| = \tfrac{1}{|S(w)|}$.
	To find a lower bound for $|z|$, we need an upper bound for $|S(w)|$. So, by triangle inequality, we have
	$$|S(w)| \le \left|\sum_{k=0}^{q-2} w^k\right| + \left|\sum_{k=0}^{p-2} w^k\right| \le \sum_{k=0}^{q-2} |w|^k + \sum_{k=0}^{p-2} |w|^k.$$
	Since $|w|<R$, the above geometric series sums are bounded as
	$$|S(w)| < \sum_{k=0}^{q-2} R^k + \sum_{k=0}^{p-2} R^k = \frac{R^{q-1}-1}{R-1} + \frac{R^{p-1}-1}{R-1} = \frac{R^{q-1}+R^{p-1}-2}{R-1}.$$
	Therefore, we obtain
	$$|z| = \frac{1}{|S(w)|} > \frac{R-1}{R^{q-1}+R^{p-1}-2}$$
	Now, combining these results, the zeros $z$ of $P(x)$ are located in the annular region
	$$ \frac{R-1}{R^{q-1}+R^{p-1}-2} < |z| < R+1, $$
	where $R=2^{1/(q-p)}$.
\end{proof}

We will illustrate Theorem \ref{zeros pq} by the examples. For $p=7$ and $q=17,$ we have
\begin{align*}
	I(\varGamma(\mathbb{Z}_{119}), x)=&1 + 22 x + 135 x^2 + 580 x^3 + 1835 x^4 + 4374 x^5 + 8009 x^6 + 
	11440 x^7 + 12870 x^8 \\
	&+ 11440 x^9 + 8008 x^{10} + 4368 x^{11} + 
	1820 x^{12} + 560 x^{13} + 120 x^{14} + 16 x^{15} + x^{16}.
\end{align*}
The zeros of the polynomial $I(\varGamma(\mathbb{Z}_{119}), x)$ are shown in  Figure \ref{zeros independent pq}(a). Also, $R = 2^{(1/(q - p))}=1.12246,$ and $\tfrac{R-1}{R^{p-1}+R^{q-1}-2}=0.0306$. Thus, by Theorem \ref{zeros pq complex}, the annular region is $ 0.0306 < |z| < 2.12246$ and the zeros of $I(\varGamma(\mathbb{Z}_{119}), x)$ lie inside it, as can be seen in Figure \ref{zeros independent pq}(b).

For $p=2,$ and $q=13,$ we have 
\begin{align*}
	I(\varGamma(\mathbb{Z}_{26}), x)=&1 + 13 x + 66 x^2 + 220 x^3 + 495 x^4 + 792 x^5 + 924 x^6 + 792 x^7 + 
	495 x^8 + 220 x^9\\
	& + 66 x^{10} + 12 x^{11} + x^{12},
\end{align*}
and its zeros are shown in Figure \ref{zeros independent pq}(c).

\begin{figure}[H]
	\centering{\scalebox{.23}{\includegraphics{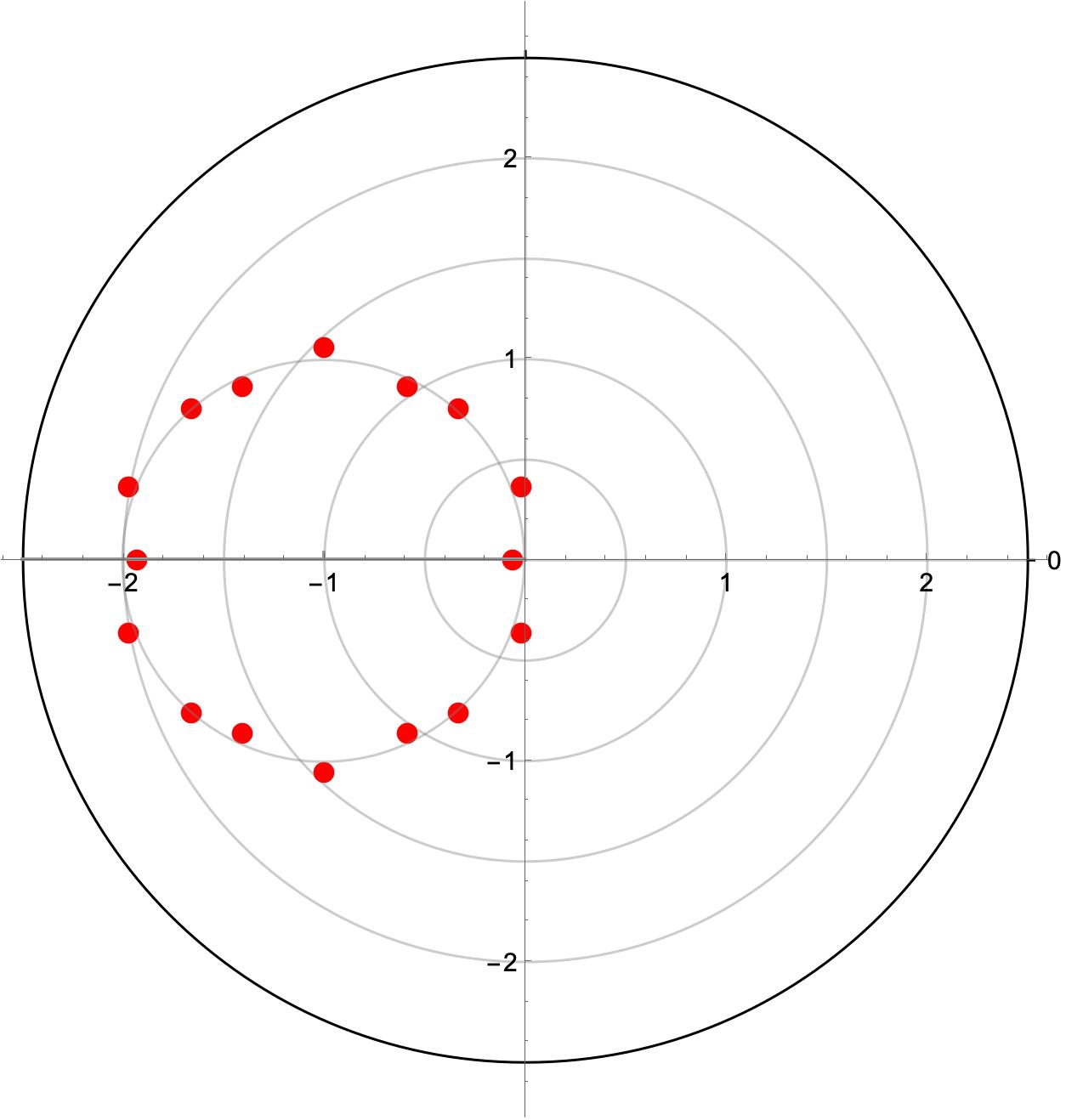}} \scalebox{.28}{\includegraphics{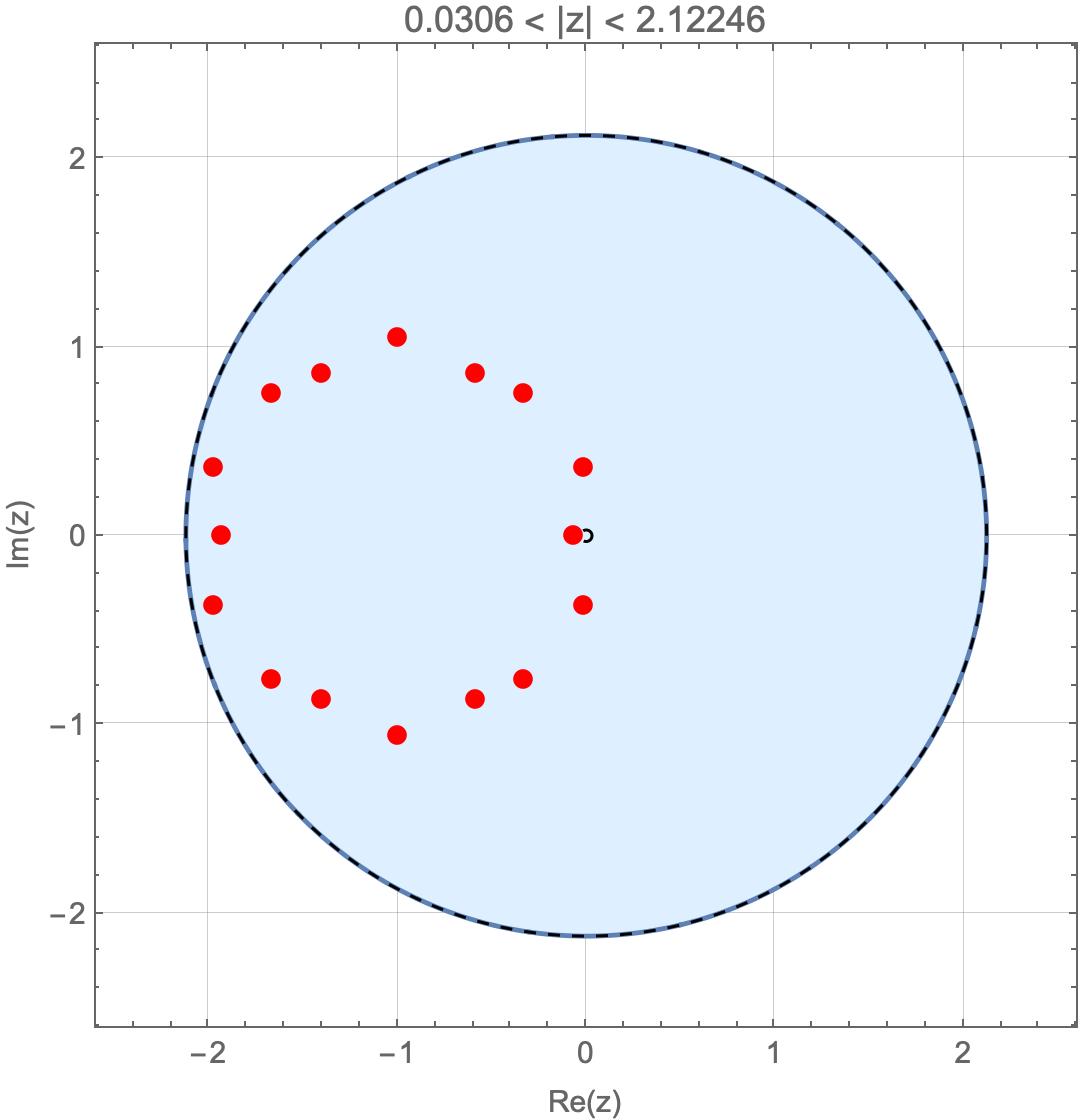}}\quad\scalebox{.23}{\includegraphics{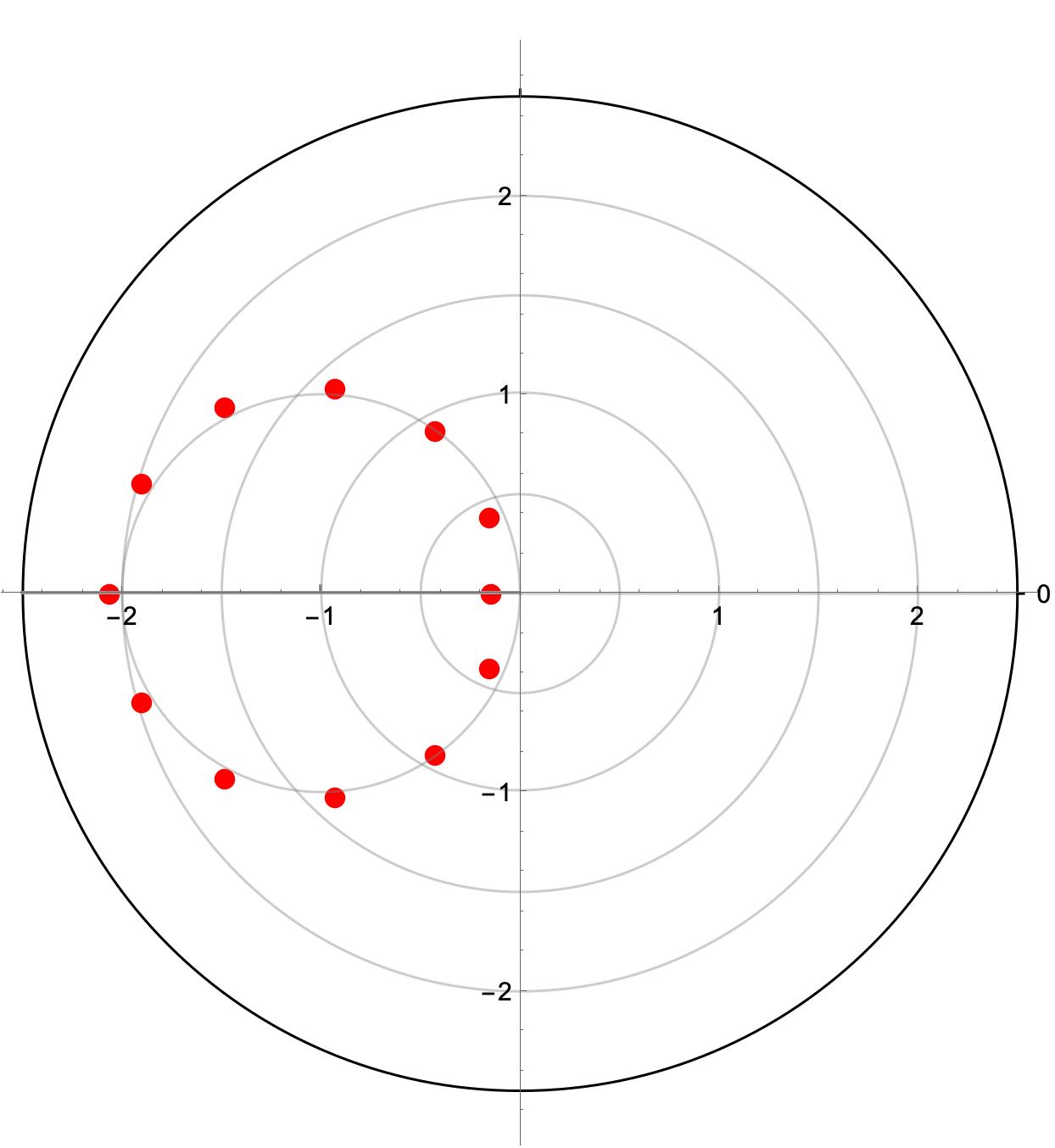}}}\\
	(a) Zeros of $I(\varGamma(\mathbb{Z}_{119}), x)$\quad  (b)  $0.0306 < |z| < 2.12246$\qquad (c) Zeros of $I(\varGamma(\mathbb{Z}_{26}), x).$
	\caption{Pictorial representation of the zeros of  $I(\varGamma(\mathbb{Z}_{119}), x)$, and $I(\varGamma(\mathbb{Z}_{26}), x)$ in a plane.}
	\label{zeros independent pq}
\end{figure}

Next result show that the independence polynomial $I(\varGamma(\mathbb{Z}_{pq}), x) $ of graph $ \varGamma(\mathbb{Z}_{pq})$
 is unimodal and log-concave. 
 \begin{theorem}
 	The polynomial $I(\varGamma(\mathbb{Z}_{pq}),x)$ is log-concave and unimodal.
 \end{theorem}
 
 \begin{proof}
 	Since the independence polynomial of $Gamma(\mathbb{Z}_{pq})$ is 
 	$
 	P(x)=I(\varGamma(\mathbb{Z}_{pq}),x)=(1+x)^n+(1+x)^m-1,
 	$
 	where $n=q-1$ and $m=p-1$. As $p<q$ are primes, so we have $1\leq m<n$. Let
 	$
 	P(x)=\sum_{k=0}^{n}\ell_kx^k.
 	$
 	Then
 	$
 	\ell_0=1,
 	$
 	and, for $1\leq k\leq n$, we have 
 	$
 	\ell_k=\binom{n}{k}+\binom{m}{k},
 	$
 	where we use the convention $\binom{m}{k}=0$ whenever $k>m$. Thus, 
 	$
 	\ell_k=
 	\binom{n}{k}+\binom{m}{k}$ for $1\leq k\leq m,$
 	and
 	$\ell_k=\binom{n}{k}$ for $m+1\leq k\leq n.$
 	We must show that
 	$
 	\ell_k^2\geq \ell_{k-1}\ell_{k+1}$ for $1\leq k\leq n-1.$
 	For  $k=1$,  we have 
 	$
 	\ell_1^2-\ell_0\ell_2
 	=
 	\tfrac{n^2+m^2+4mn+n+m}{2}>0,
 	$ and hence the log-concavity inequality holds for $k=1$.
 	 Next, assume that $2\leq k\leq m$. If 
 	$
 	N_j=\binom{n}{j},$ and $ M_j=\binom{m}{j},
 	$
 	then
 	$
 	\ell_j=N_j+M_j
 	$
 	for $j=k-1,k$, while the same formula also holds for $j=k+1$, if we agree that $M_{m+1}=0$ when $k=m$. We will prove that
 	$$
 	(N_k+M_k)^2\geq (N_{k-1}+M_{k-1})(N_{k+1}+M_{k+1}), 
 	$$
 	or equivalently, by expanding, it is enough to prove that 
 	$
 	N_k^2\geq N_{k-1}N_{k+1}, 
 	M_k^2\geq M_{k-1}M_{k+1},
 	$
 	and
 	$
 	2N_kM_k\geq N_{k-1}M_{k+1}+M_{k-1}N_{k+1}.
 	$
 	The first two inequalities are the standard log-concavity inequalities for binomial coefficients. It remains to prove the mixed type  third inequality.  As
 	$
 	\binom{r}{k}
 	=
 	\frac{r-k+1}{k}\binom{r}{k-1}
 	$
 	and
 	$
 	\binom{r}{k+1}
 	=
 	\frac{r-k}{k+1}\binom{r}{k},
 	$
 	the mixed inequality becomes
 	$$
 	2\binom{n}{k}\binom{m}{k}
 	\geq
 	\binom{n}{k-1}\binom{m}{k+1}
 	+
 	\binom{m}{k-1}\binom{n}{k+1}.
 	$$
 	Thus, by substituting their values, we have
 	\begin{align*}
 		2\frac{n-k+1}{k}\binom{n}{k-1}
 		\frac{m-k+1}{k}\binom{m}{k-1}\geq&
 		\binom{n}{k-1}
 		\frac{m-k}{k+1}
 		\frac{m-k+1}{k}
 		\binom{m}{k-1} \\ & +
 		\binom{m}{k-1}
 		\frac{n-k}{k+1}
 		\frac{n-k+1}{k}
 		\binom{n}{k-1}.
 	\end{align*}
 	Since $2\leq k\leq m$, the factors involved are non-negative, and after simplifying, 
 	we obtain
 	$$
 	2(k+1)(n-k+1)(m-k+1)
 	\geq
 	k(m-k)(m-k+1)+k(n-k)(n-k+1).
 	$$
 	With compact notations
 	$
 	A=n-k+1,$ and $ B=m-k+1.
 	$, we have
 	$$
 	2(k+1)AB\geq k(B-1)B+k(A-1)A
 	$$
 	which is equivalent to 
 	$$
 	2(k+1)\geq k\left(1-\frac{1}{B}\right)+k\left(1-\frac{1}{A}\right).
 	$$
 	The right-hand side of above inequality is at most $2k$, while the left-hand side is $2k+2$. Hence the inequality holds, and $
 	\ell_k^2\geq \ell_{k-1}\ell_{k+1}
 	$
 	for every $2\leq k\leq m$. For the boundary index $k=m+1$, which occurs only when $m+1\leq n-1$, that is, when $n\geq m+2$. So,  $ \ell_{m+1}^{2}\geq \ell_{m}\ell_{m+2}$ gives  
 	$$
 	\binom{n}{m+1}^2
 	\geq
 	\left(\binom{n}{m}+1\right)\binom{n}{m+2}.
 	$$
 	If $
 	a=n-m, $ then $a\geq 2$ as  $n\geq m+2$. Thus, with 
 	$
 	\binom{n}{m+1}
 	=
 	\binom{n}{m}\frac{a}{m+1}
 	$
 	and
 	$
 	\binom{n}{m+2}
 	=
 	\binom{n}{m}\frac{a(a-1)}{(m+1)(m+2)},
 	$
 	the  inequality is equivalent to
 	$
 	\binom{n}{m}(n+1)\geq (a-1)(m+1).
 	$
 	Now $1\leq m\leq n-2$, so $\binom{n}{m}\geq n$, $
 	(a-1)+(m+1)=n,
 	$
 	and hence
 	$$
 	(a-1)(m+1)\leq \frac{n^2}{4}\leq n(n+1)\leq \binom{n}{m}(n+1).
 	$$
 	Therefore, the log-concavity inequality also holds at $k=m+1$. Finally, suppose $k\geq m+2$. Then
 	$
 	\ell_{k-1}=\binom{n}{k-1},
 	\ell_k=\binom{n}{k},$ and $
 	\ell_{k+1}=\binom{n}{k+1}.
 	$
 	Thus,
 	$
 	\ell_k^2\geq \ell_{k-1}\ell_{k+1}
 	$
 	follows directly from the log-concavity of the binomial coefficients in the expansion of $(1+x)^n$.
 	 Combining all cases, the coefficient sequence
 	$
 	(\ell_0,\ell_1,\ldots,\ell_n)
 	$
 	is log-concave. Since all coefficients are positive, the sequence has no internal zeros. A positive log-concave sequence is unimodal \cite{stanley}. Therefore, $P(x)=I(\varGamma(\mathbb{Z}_{pq}),x)$ is both log-concave and unimodal.
 \end{proof}

\section{Independence polynomial of zero divisor graph of $ \mathbb{Z}_{p^2q}$}\label{section 3}
The following result gives the independent  polynomial of the zero divisor graph $\varGamma(\mathbb{Z}_{p^2q})$.
\begin{theorem}\label{ind poly zp^2q}
	For primes $p<q$, the independence polynomial of $\varGamma(\mathbb{Z}_{p^2q})$ is
	$$I(\varGamma(\mathbb{Z}_{p^2q}); x) = (1+x)^{(p-1)(p+q-1)} + (1+x)^{p(q-1)} - (1+x)^{(p-1)(q-1)} + (p-1)x(1+x)^{p(p-1)}.$$
\end{theorem}
\begin{proof}
	Let $G\cong \varGamma(\mathbb{Z}_{p^2q})$ be the zero divisor graph of ring $\mathbb{Z}_{p^2q}.$ The vertex set $V$ can be partitioned into four disjoint sets based on the greatest common divisor (gcd) with $n=p^2q$, that is,
	$V_d = \{k \in \{1, \dots, n-1\} \mid \gcd(k, n) = d\}$. The non-zero zero divisors correspond to $d \in \{p, p^2, q, pq\}$, and their cardinalities are:
	\begin{align*}
		|V_p| &= \varphi(n/p) = \varphi(pq) = (p-1)(q-1), |V_{p^2}| = \varphi(n/p^2) = \varphi(q) = q-1,\\
		|V_q| &= \varphi(n/q) = \varphi(p^2) = p(p-1), |V_{pq}| = \varphi(n/pq) = \varphi(p) = p-1.
	\end{align*}
	For each $d\in \{p,q,p^{2}\},$ the set $V_{d} $ is an independent set as $d^{2}$ is not multiple of $n$. Also, the $V_{pq}$ is a clique, since $(pq)^{2}$ is multiple of $n$. Furthermore, each vertex of $V_{pq}$ is adjacent to every vertex of $V_{p}$ and $V_{p^{2}}$, and each vertex of $V_{q}$ is adjacent to every vertex of $V_{p^{2}}.$  
	Let $\mathcal{I}$ be the set of all independent sets of $G$. We partition $\mathcal{I}$ into two disjoint collections:
	$$\mathcal{I}_1 = \{S \in \mathcal{I} \mid S \cap V_{pq} = \emptyset\}\quad \text{and}\quad \mathcal{I}_2 = \{S \in \mathcal{I} \mid S \cap V_{pq} \neq \emptyset\}.$$ 
	Based on the above decomposition, the independence polynomial of $G$ can be written as: 
	$$I(G, x) = \sum_{S \in \mathcal{I}_1} x^{|S|} + \sum_{S \in \mathcal{I}_2} x^{|S|}.$$
	First consider the set $\mathcal{I}_1$. In this case, the independent sets of $G$ are the independent sets of the subgraph $G' = \varGamma[V_p \cup V_{p^2} \cup V_q]$. As there are no edges between $V_p$ and $V_{p^2} \cup V_q$, so $G'$ is the disjoint union of $\varGamma[V_p]$ and $\varGamma[V_{p^2} \cup V_q]$. We recall that  the independence polynomial of a disjoint union of graphs is the product of their independence polynomials.
	So, the independence polynomial of  $\varGamma[V_p]\cong \overline{K}_{|V_{p}|}$ is $I(\varGamma[V_p]; x) = (1+x)^{|V_p|}$. By \eqref{eq join}, the independence polynomial of 
	$\varGamma[V_{p^2} \cup V_q]\cong K_{|V_{p^2}|, |V_q|}$ is $I(\varGamma[V_{p^2} \cup V_q]; x) = (1+x)^{|V_{p^2}|} + (1+x)^{|V_q|} - 1$.
	Thus, the total contribution from $\mathcal{I}_1$ is $$I(G', x) = (1+x)^{|V_p|}\left((1+x)^{|V_{p^2}|} + (1+x)^{|V_q|} - 1\right).$$
	Next, consider $S \in \mathcal{I}_2$. As $V_{pq}$ is a clique, and it follows that $S$ can contain at most one vertex from $V_{pq}$. So, $S \cap V_{pq} \neq \emptyset$, and $|S \cap V_{pq}|=1$. Assume that  $S \cap V_{pq} = \{v\}$ for some $v \in V_{pq}$, and by structure of $G,$ we note that
	$v$ is adjacent to every vertex in $V_p$ and $V_{p^2}$. Thus, we obtain $S \cap V_p = \emptyset$ and $S \cap V_{p^2} = \emptyset$. As there are no edges between $V_{pq}$ and $V_q$, so the set, $S \setminus \{v\}$ may be any subset of $V_q$.
	Thus, any set in $\mathcal{I}_2$ is of the form $\{v\} \cup S_q$ where $v \in V_{pq}$ and $S_q \subseteq V_q$. With this information, the generating function for such sets is:
	\begin{align*}
		\sum_{v \in V_{pq}} \sum_{S_q \subseteq V_q} x^{|\{v\} \cup S_q|} &= \sum_{v \in V_{pq}} \sum_{S_q \subseteq V_q} x^{1+|S_q|} = |V_{pq}| \sum_{k=0}^{|V_q|} \binom{|V_q|}{k} x^{1+k} \\ 
		&= |V_{pq}|x \sum_{k=0}^{|V_q|} \binom{|V_q|}{k} x^k = |V_{pq}|x(1+x)^{|V_q|}.
	\end{align*}
	Summing the contributions from $\mathcal{I}_{1}$ and $\mathcal{I}_{2}$, and substituting the values $|V_p| = (p-1)(q-1), |V_{p^2}| = q-1,|V_q| = p(p-1)$, and $|V_{pq}| = p-1$, we obtain the independence polynomial of $G$ as
	\begin{align*}
		I(G; x) =& (1+x)^{(p-1)(q-1)}\left((1+x)^{q-1} + (1+x)^{p(p-1)} - 1\right) + (p-1)x(1+x)^{p(p-1)}\\
		=& (1+x)^{(p-1)(q-1)+q-1} + (1+x)^{(p-1)(q-1)+p(p-1)} - (1+x)^{(p-1)(q-1)}\\
		& + (p-1)x(1+x)^{p(p-1)}\\
		=& (1+x)^{p(q-1)} + (1+x)^{(p-1)(p+q-1)} - (1+x)^{(p-1)(q-1)} + (p-1)x(1+x)^{p(p-1)}.
	\end{align*}
\end{proof}

For $p=3$ and $q=5$, the independence polynomial of $G\cong \varGamma(\mathbb{Z}_{45})$ is
$$I(G, x) = (1+x)^{14} + (1+x)^{12} - (1+x)^8 + 2x(1+x)^6.$$
We give detailed information for this example. The set of vertices $V(G)$ can be partitioned based on the greatest common divisor of an element with $45.$ Let $V_d = \{v \in V \mid \gcd(v, 45) = d\}$. The possible values for $d$ are the divisors of $45$ other than $1$ and $45,$ which are $3, 5, 9, 15.$ The sizes of the vertex partitions $V_{d}$ are as: $|V_3| = \varphi(45/3) = \varphi(15) = 8, |V_5| = \varphi(45/5) = \varphi(9) = 6, |V_9| = \varphi(45/9) = \varphi(5) = 4$, and  $|V_{15}| = \varphi(45/15) = \varphi(3) = 2$.
The order $G$ is $|V| = 8+6+4+2=20$.  Also, $V_{15}=\{15,30\}, V_{5}=\{5,10,20,25,35,40\},V_{9}=\{9,18,27,36\}$ and $V_{3}=\{3,6,12,21,24,33,39,42\}.$
The sets $V_3, V_5, V_9$ are independent sets, and the set $V_{15}$ induces a $K_2$. Every vertex in $V_{15}$ is adjacent to every vertex in $V_3 \cup V_9$, and every vertex in $V_{5}$ is adjacent to every vertex in $V_{9}.$ The graph $G$ is shown in Figure \ref{zero divisor graph z36} with green vertices in $V_{15}$, blue in $V_{9}$, and white in  $V_{5}, $ and $V_{3}$. 
\begin{figure}[H]
\centering
	\begin{tikzpicture}[
		v/.style={circle, draw, inner sep=2pt, minimum size=20pt}
		]
		\def\Sfive{5,10,20,25,35,40}      
		\def\Sthreetwo{9,18,27,36}        
		\def\Sfifteen{15,30}             
		\def\Sthreeone{3,6,12,21,24,33,39,42} 
		
		\foreach \name [count=\i] in \Sfive {
			\node[v] (\name) at (-6, 5.25 - 1.5*\i) {\name};
		}
		\foreach \name [count=\i] in \Sthreetwo {
			\node[v, fill=cyan!25] (\name) at (-2, 4.5 - 2*\i) {\name};
		}
		\foreach \name [count=\i] in \Sfifteen {
			\node[v, fill=green!25] (\name) at (2, 2.25 - 3*\i) {\name};
		}
		\foreach \name [count=\i] in \Sthreeone {
			\node[v] (\name) at (6, 5.2 - 1.2*\i) {\name};
		}
		
		\foreach \sfive in \Sfive {
			\foreach \sthreetwo in \Sthreetwo {
				\draw[gray] (\sfive) -- (\sthreetwo);
			}
		}
		\foreach \sthreetwo in \Sthreetwo {
			\foreach \sfifteen in \Sfifteen {
				\draw[blue] (\sthreetwo) -- (\sfifteen);
			}
		}
		\foreach \sthreeone in \Sthreeone {
			\foreach \sfifteen in \Sfifteen {
				\draw[red] (\sthreeone) -- (\sfifteen);
			}
		}
		\draw[thick, green!50!black] (15) -- (30);
		
	\end{tikzpicture}
	\caption{Zero divisor graph of $\mathbb{Z}_{45}.$}
	\label{zero divisor graph z36}
\end{figure}
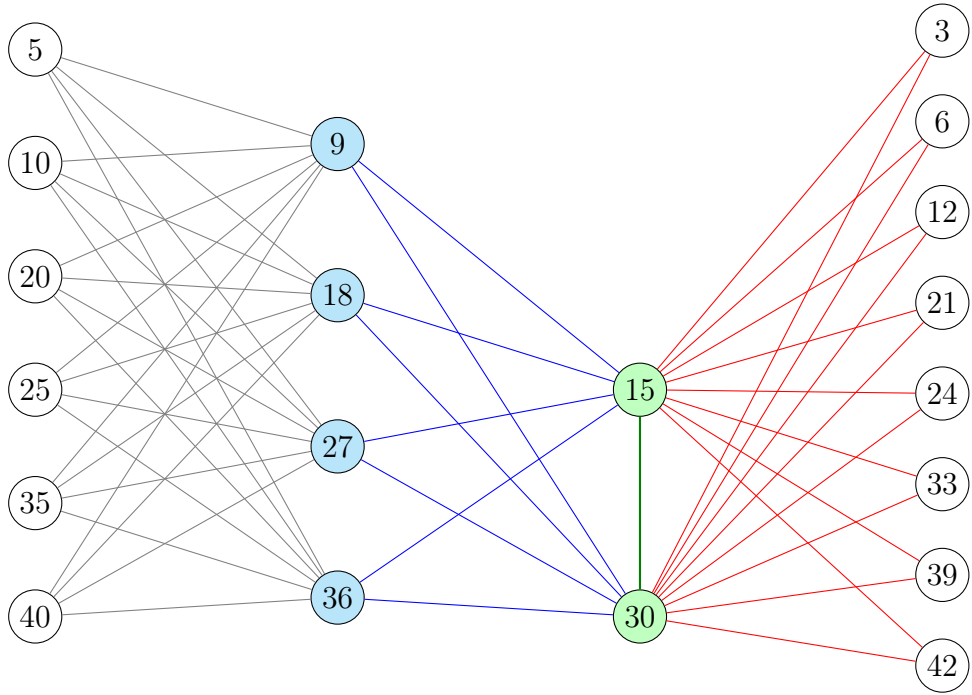

We compute $I(G,x)$ by considering the composition of an arbitrary independent set $S \subseteq V$ with respect to the partition $V_{15}$. With $V_{15}=\{v_a, v_b\}$, and consider the case $S \cap V_{15} = \emptyset$. 
An independent set $S$ is a subset of $V_3 \cup V_5 \cup V_9$. The subgraph induced by these vertices, $G[V_3 \cup V_5 \cup V_9]$, consists of the disjoint union of an empty graph on $V_3$ (8 vertices) and the complete bipartite graph $K_{6,4}$ on $V_5 \cup V_9$. The independence polynomial for this case is the product of the polynomials for these two components:
$$I(\overline{K}_8, x) \cdot I(K_{6,4}, x) = (1+x)^8 \left((1+x)^6 + (1+x)^4 - 1\right).$$

Next, we consider $S \cap V_{15} \neq \emptyset$. As $v_a$ and $v_b$ are adjacent, an independent set $S$ can contain at most one of them.
If $v_a \in S$, then $S$ cannot contain any neighbor of $v_a$. The neighborhood of $v_a$ is $V_3 \cup V_9 \cup \{v_b\}$. Thus, $S$ must be of the form $\{v_a\} \cup S'$, where $S'$ is an independent set in the remaining vertices, which is $V_5$. Since $V_5$ is an independent set of size $6$, so $S'$ can be any subset of $V_5$. The generating function for such sets $S$ is $x \cdot \sum_{k=0}^6 \binom{6}{k}x^k = x(1+x)^6$.
Again, if $v_b \in S$, with the same idea, the neighborhood of $v_b$ is $ V_3 \cup V_9 \cup \{v_a\}$, and the contribution such subsets to the independence polynomial of $G$ is also $x(1+x)^6$. Summing, all these case, we get the independence polynomial of $G$ as given below
\begin{align*}
	I(G,x) =& (1+x)^8 \left((1+x)^6 + (1+x)^4 - 1\right)+ x(1+x)^6 + x(1+x)^6\\
	=& (1+x)^8 (1+x)^6 + (1+x)^8 (1+x)^4 - (1+x)^8 + 2x(1+x)^6\\
	=& (1+x)^{14} + (1+x)^{12} - (1+x)^8 + 2x(1+x)^6.
\end{align*}

Next, we show that independence polynomial of $\varGamma(\mathbb{Z}_{p^2q})$ is log-concave and unimodal.
\begin{theorem}
	For primes $p<q$, he independence polynomial
	$$
	I(\varGamma(\mathbb{Z}_{p^2q});x)
	=
	(1+x)^{(p-1)(p+q-1)}
	+
	(1+x)^{p(q-1)}
	-
	(1+x)^{(p-1)(q-1)}
	+
	(p-1)x(1+x)^{p(p-1)}
	$$
	is log-concave and  unimodal.
\end{theorem}

\begin{proof}
	For brevity,  we let
	$a=(p-1)(q-p-1),  b=q-1, c=p(p-1),$ and $ d=p-1.$
	Since $p<q$, we have $a\geq 0$, $b\geq 2$, $c\geq 2$, and $d\geq 1$. It is easy to see that
	$
	I(\varGamma(\mathbb{Z}_{p^2q});x)
	=
	(1+x)^c H_{p,q}(x),
	$
	where
	$$
	H_{p,q}(x)
	=
	(1+x)^a\left((1+x)^b+(1+x)^c-1\right)+dx.
	$$
	Clearly,
	$
	c+a+c=(p-1)(p+q-1), 
	c+a+b=p(q-1),
	$
	and
	$
	c+a=(p-1)(q-1).
	$
	 We shall use the following standard fact: if two polynomials have non-negative log-concave coefficient sequences with no internal zeros, then their product also has a log-concave coefficient sequence. This is the usual convolution property of log-concave sequences \cite{stanley}. Further, we note that 
	$
	(1+x)^m+(1+x)^n-1
	$
	has a log-concave coefficient sequence for all positive integers $m,n$. Its coefficients are
	$$
	1,\quad \binom{m}{1}+\binom{n}{1},\quad \binom{m}{2}+\binom{n}{2},\quad \ldots,
	$$
	where $\binom{r}{k}=0$ for $k>r$, and the log-concavity follows by the usual binomial-ratio verification.
	Let 
	$
	J_{p,q}(x)=(1+x)^a\left((1+x)^b+(1+x)^c-1\right).
	$
	The polynomial $(1+x)^a$ is log-concave, and by the preceding observation, the polynomial $(1+x)^b+(1+x)^c-1$ is log-concave.  Hence, $J_{p,q}(x)$ is log-concave. Let
	$
	J_{p,q}(x)=\sum_{k\geq 0}u_kx^k
	$
	and
	$
	H_{p,q}(x)=\sum_{k\geq 0}h_kx^k.
	$
	Since $H_{p,q}(x)=J_{p,q}(x)+dx$, we have
	$
	h_0=u_0=1, h_1=u_1+d,$ and $ h_k=u_k$ for $k\geq 2.$
	Therefore, all log-concavity inequalities for $H_{p,q}(x)$ are inherited from $J_{p,q}(x)$ except possibly the inequalities centered at $h_1$ and $h_2$. Thus, it remains to verify
	$
	h_1^2\geq h_0h_2
	$
	and
	$
	h_2^2\geq h_1h_3.
	$
	 Since
	$
	H_{p,q}(x)
	=
	(1+x)^a\left((1+x)^b+(1+x)^c-1\right)+dx,
	$
	we get
	$
	h_0=1, $ and $h_1=a+b+c+d=pq-1.$ For $k\geq 2$, the term $dx$ does not contribute, so
	$
	h_k=\binom{a+c}{k}+\binom{a+b}{k}-\binom{a}{k}.
	$
	In particular,
	$$
	h_2
	=
	\binom{(p-1)(q-1)}{2}
	+
	\binom{p(q-p)}{2}
	-
	\binom{(p-1)(q-p-1)}{2}=\frac{p(q-1)(pq-3p+1)}{2}.
	$$
	Similarly,
	$$
	h_3
	=
	\binom{(p-1)(q-1)}{3}
	+
	\binom{p(q-p)}{3}
	-
	\binom{(p-1)(q-p-1)}{3},
	$$
	and hence
	$$
	h_3
	=
	\frac{p(q-1)}{6}
	\left(
	3p^3+p^2q^2-8p^2q+p^2+6pq+3p-3q-1
	\right).
	$$
	We first prove $h_1^2\geq h_0h_2$, and with $h_0=1$, it is equivalent to $h_1^2\geq h_2$. Using the expressions above, we have
	$$
	h_1^2-h_2
	=
	(pq-1)^2-\frac{p(q-1)(pq-3p+1)}{2}=
	\frac{p^2q^2+4p^2q-3p^2-5pq+p+2}{2}.
	$$
	To see that this is positive, let $q=p+s$, where $s\geq 1$, and let  $p=u+2, s=v+1$, where $u,v\geq 0$. Thus, 
	$$
	2(h_1^2-h_2)
	=
	u^4+2u^3v+14u^3+u^2v^2+18u^2v+57u^2
	+4uv^2+43uv+88u+4v^2+30v+46.
	$$
	Every term on the right-hand side is non-negative, and the constant term is positive. So, we obtain 
	$
	h_1^2-h_0h_2=h_1^2-h_2>0.
	$	
	It remains to prove $h_2^2\geq h_1h_3$. With values of $h_1$, $h_2$, and $h_3$, we obtain
	$$
	h_2^2-h_1h_3
	=
	\frac{p(q-1)}{12}\Phi(p,q-p),
	$$
	where, if $s=q-p$, then
	$$
	\begin{aligned}
		\Phi(p,s)=&
		p^6+3p^5s-11p^5+3p^4s^2-16p^4s+39p^4+p^3s^3-5p^3s^2  \\
		&+35p^3s-61p^3-4p^2s^2-34p^2s+37p^2+6ps^2+17ps-3p-6s-2.
	\end{aligned}
	$$
	Let $p=u+2$ and $s=v+1$, where $u,v\geq 0$. Then
	$$
	\begin{aligned}
		\Phi(p,s)=&
		u^6+3u^5v+4u^5+3u^4v^2+20u^4v+6u^4+u^3v^3+22u^3v^2\\
		&+68u^3v+18u^3+6u^2v^3+56u^2v^2+126u^2v+43u^2\\
		&+12uv^3+62uv^2+117uv+40u+8v^3+28v^2+44v+12.
	\end{aligned}
	$$
	All coefficients in the above expansion are non-negative, and the constant term is positive. Therefore, $\Phi(p,s)>0$. Since $p(q-1)>0$, it follows that
	$
	h_2^2-h_1h_3>0,
	$, and hence 
	Thus
	$
	h_2^2\geq h_1h_3.
	$
	Thus, it follows that $H_{p,q}(x)$ is log-concave. Finally, since $(1+x)^c$ is log-concave and has no internal zeros, and  $H_{p,q}(x)$ is log-concave with positive coefficients, so their product
	$
	I(\varGamma(\mathbb{Z}_{p^2q});x)=(1+x)^cH_{p,q}(x)
	$
	is log-concave. Hence, the coefficient sequence of $I(\varGamma(\mathbb{Z}_{p^2q});x)$ is log-concave. Moreover, all coefficients of $I(\varGamma(\mathbb{Z}_{p^2q});x)$ are positive. A positive log-concave sequence has no internal zeros and is unimodal \cite{stanley}. Therefore, $I(\varGamma(\mathbb{Z}_{p^2q});x)$ is unimodal.
\end{proof}

Clearly, $x=-1$ is a zero of the polynomial $I(\varGamma(\mathbb{Z}_{p^2q}), x)$. For other zeros the independence polynomial of $\varGamma(\mathbb{Z}_{p^2q}),$ we have the following result, which gives their annular region.
\begin{theorem}
	Let $I(x)$ be the independence polynomial of $\varGamma(\mathbb{Z}_{p^{2}q})$, where $p<q$ are primes. Then  all zeros of $I(x)$ different from $x=-1$ lie in the annular region
	$$
	\left\{x\in\mathbb{C}:\frac{1}{2}<|1+x|<2\right\}.
	$$
\end{theorem}

\begin{proof}
	Since the independence polynomial of $\varGamma(\mathbb{Z}_{p^{2}q})$ is
	$$
	I(x)
	=
	(1+x)^{(p-1)(p+q-1)}
	+
	(1+x)^{p(q-1)}
	-
	(1+x)^{(p-1)(q-1)}
	+
	(p-1)x(1+x)^{p(p-1)}.
	$$
	With 
	$
	y=1+x,
	$ the above polynomial becomes
	$$
	I(y-1)
	=
	y^{(p-1)(p+q-1)}
	+
	y^{p(q-1)}
	-
	y^{(p-1)(q-1)}
	+
	(p-1)(y-1)y^{p(p-1)}.
	$$
	Since $p(p-1)>0$, it follows immediately that $y=0$, that is $x=-1$, is a zero of $I(x)$. We are interested only in the other zeros, so we assume $y\neq 0$.  The polynomial can be written as $
	I(y-1)
	=
	y^{p(p-1)}Q(y),
	$
	where
	$$
	Q(y)
	=
	y^{A}+y^{B}-y^{C}+(p-1)y-(p-1),
	$$
	with
	$
	A=(p-1)(q-1),
	B=p(q-p),$ and $
	C=(p-1)(q-p-1).
	$
	Thus the zeros of $I(x)$ different from $x=-1$ correspond exactly to the non-zero zeros of $Q(y)$. It is therefore enough to prove that every zero $y$ of $Q(y)$ satisfies
	$
	\tfrac{1}{2}<|y|<2.
	$
	We shall prove separately that $Q(y)\neq 0$ for $|y|\leq \tfrac12$ and for $|y|\geq 2$.
	 First suppose that $|y|\leq \tfrac12$, and we will show that $Q(y)\neq 0$. With $r=|y|$, if $Q(y)=0$, then
	$$
	(p-1)
	=
	y^{A}+y^{B}-y^{C}+(p-1)y.
	$$
	Now, with absolute values, we have
	$$
	p-1
	\leq
	r^{A}+r^{B}+r^{C}+(p-1)r.
	$$
	Next, we check that the right-hand side of above inequality is always strictly smaller than $p-1$.
	 If $(p,q)=(2,3)$, then
	$
	Q(y)=2y^{2}+y-2.
	$
	If $|y|\leq \tfrac12$ and $Q(y)=0$, then
	$$
	2=|2y^{2}+y|\leq 2|y|^{2}+|y|\leq 2\left(\frac12\right)^2+\frac12=1,
	$$
	which is impossible. Next, let $p=2$ and $q\geq 5$. Then
	$
	Q(y)=y^{q-1}+y^{2q-4}-y^{q-3}+y-1.
	$
	If $Q(y)=0$, then
	$
	1
	\leq
	r^{q-1}+r^{2q-4}+r^{q-3}+r.
	$
	Since $q\geq 5$ and $r\leq \frac12$, we obtain
	$$
	1
	\leq
	\left(\frac12\right)^4
	+
	\left(\frac12\right)^6
	+
	\left(\frac12\right)^2
	+
	\frac12
	=
	\frac{53}{64}<1,
	$$
	a contradiction. For $p=3$, since $p<q$ and both $3$ and $q$ are primes, so we have $q\geq 5$. Hence,
	$$
	A=2(q-1)\geq 8,\qquad
	B=3(q-3)\geq 6,\qquad
	C=2(q-4)\geq 2.
	$$
	Thus, if $Q(y)=0$ and $|y|\leq \frac12$, then
	$$
	2
	\leq
	\left(\frac12\right)^8
	+
	\left(\frac12\right)^6
	+
	\left(\frac12\right)^2
	+
	2\cdot\frac12
	<2,
	$$
	which is again impossible. For $p\geq 5$, since $p$ and $q$ are odd primes with $p<q$, so we have $q\geq p+2$. Hence, 
	$
	C=(p-1)(q-p-1)\geq p-1\geq 4.
	$
	Also $A\geq 4$ and $B\geq 4$. Therefore, we derive
	$$
	r^{A}+r^{B}+r^{C}+(p-1)r
	\leq
	3\left(\frac12\right)^4+\frac{p-1}{2}
	=
	\frac{3}{16}+\frac{p-1}{2}.
	$$
	Since $p\geq 5$, we have $p-1\geq 4$, and so
	$
	\frac{3}{16}+\frac{p-1}{2}<p-1.
	$
	This contradicts the inequality obtained from $Q(y)=0$. Hence, $Q(y)$ has no zeros in $|y|\leq \frac12$.
	
	It remains to show that $Q(y)$ has no zeros in $|y|\geq 2$. Let $r=|y|\geq 2$. We use the relations
	$
	A-C=p(p-1), 
	B-C=q-1,
	$
	and
	$
	A-B=p(p-1)-(q-1).
	$
	We have to consider three cases. If $A>B$. Then $A-B\geq 1$. If $Q(y)=0$, then
	$
	y^A=-y^B+y^C-(p-1)y+(p-1).
	$
	With absolute values, we get
	$$
	1
	\leq
	r^{B-A}+r^{C-A}+(p-1)r^{1-A}+(p-1)r^{-A}.
	$$
	Since $A-C=p(p-1)$, and $r\geq 2$, this implies that
	$$
	1
	\leq
	2^{-(A-B)}
	+
	2^{-p(p-1)}
	+
	(p-1)2^{-(A-1)}
	+
	(p-1)2^{-A}.
	$$
	For case $A>B$, we necessarily have $p\geq 3$. Since $q\geq p+2$, we obtain
	$
	A=(p-1)(q-1)\geq (p-1)(p+1),
	$
	and hence $A-1\geq p(p-1)$. Thus, we obtain
	$$
	1
	\leq
	\frac12
	+
	2^{-p(p-1)}
	+
	(p-1)2^{-p(p-1)}
	+
	(p-1)2^{-p(p-1)-1},
	$$
	or equivalently,
	$$
	1
	\leq
	\frac12+
	\frac{1+\frac32(p-1)}{2^{p(p-1)}}.
	$$
	But $p\geq 3$, so $p(p-1)\geq 6$, and hence
	$$
	\frac{1+\frac32(p-1)}{2^{p(p-1)}}<\frac12.
	$$
	Therefore, the right-hand side is strictly less than $1$, a contradiction. Hence $Q(y)\neq 0$ for $|y|\geq 2$ in the case $A>B$.
	 Now, for $A=B$, we get 
	$
	Q(y)=2y^A-y^C+(p-1)y-(p-1).
	$
	If $Q(y)=0$, then
	$
	2r^A\leq r^C+(p-1)r+(p-1), 
	$ and thereby we have
	$$
	2\leq r^{C-A}+(p-1)r^{1-A}+(p-1)r^{-A}.
	$$
	Since $A-C=p(p-1)$ and $r\geq 2$, the right-hand side of above inequality  is at most
	$$
	2^{-p(p-1)}+(p-1)2^{1-A}+(p-1)2^{-A},
	$$
	which is strictly less than $2$. Clearly, for $(p,q)=(2,3)$, it is at most
	$
	\frac14+\frac12+\frac14=1,
	$
	and for $p\geq 3$ it is even smaller, since $A$ and $p(p-1)$ are larger. Thus, we get a contradiction. Therefore $Q(y)\neq 0$ for $|y|\geq 2$ in the case $A=B$. Finally, if $B>A$, then
	$
	B-A=(q-1)-p(p-1)>0.
	$
	Since $q$ is odd and $p(p-1)$ is even, the positive integer $B-A$ is in fact at least $2$. If $Q(y)=0$, then
	$$
	y^B=-y^A+y^C-(p-1)y+(p-1),
	$$
	and hence
	$$
	1
	\leq
	r^{A-B}+r^{C-B}+(p-1)r^{1-B}+(p-1)r^{-B}.
	$$
	With $B-C=q-1$, we obtain
	$$
	1
	\leq
	2^{-(B-A)}
	+
	2^{-(q-1)}
	+
	(p-1)2^{1-B}
	+
	(p-1)2^{-B}.
	$$
	Since $B-A\geq 2$ and $q-1\geq 4$, the first two terms are bounded by $\tfrac14$ and $\tfrac1{16}$, respectively. Also, we note that 
	$
	(p-1)2^{1-B}+(p-1)2^{-B}
	=
	\tfrac{3(p-1)}{2^B}.
	$
	If $p=2$, then $q\geq 5$, so $B=2(q-2)\geq 6$, and therefore
	$
	\tfrac{3(p-1)}{2^B}\leq \tfrac{3}{64}.
	$
	If $p\geq 3$, then $B>A$ implies that $q-1>p(p-1)$, and so $B=p(q-p)$ is sufficiently large. In particular,
	$
	\tfrac{3(p-1)}{2^B}<\tfrac14.
	$
	Thus,  in all cases we get
	$
	1
	<
	\tfrac14+\tfrac1{16}+\tfrac14
	<1,
	$
	which is absurd. Therefore, $Q(y)$ has no zeros with $|y|\geq 2$ in the case $B>A$.
	Thus, by combining the two parts, every zero $y$ of $Q(y)$ satisfies
	$
	\tfrac12<|y|<2.
	$
	Since the zeros of $I(x)$ different from $x=-1$ correspond to the zeros of $Q(y)$ under the change of variable $y=1+x$, we conclude that every zero $x\neq -1$ of $I(x)$ satisfies
	$
	\tfrac12<|1+x|<2.
	$
\end{proof}

For $p=3$ and $q=5,$ the independent of $G\cong \varGamma(\mathbb{Z}_{p^2q})$ is
\begin{align*}
	I(G,x) &= 2 x (1 + x)^6 - (1 + x)^8 + (1 + x)^{12} + (1 + x)^{14},\\
	&=(1 + x)^6 (1 + 14 x + 42 x^2 + 76 x^3 + 85 x^4 + 62 x^5 + 29 x^6 + 8 x^7 + x^8).
\end{align*}
and its zeros are \begin{align*}
	(-1)^{6},-2.1584,-0.09376,-1.66756\pm1.00486 i, -0.8428\pm 1.21902 i, -0.3634\pm0.679226 i.
\end{align*}
All the zeros of $I(G,x) $ are shown in (a) Figure \ref{zeros independent 45}. The zeros of $I(G,x) $ other than $x=-1,$ are shown in (b) Figure \ref{zeros independent 45}, and clearly they lie in annular region $\tfrac{1}{2}<|1+x|<2. $
\begin{figure}[H]
	\centering{\scalebox{.35}{\includegraphics{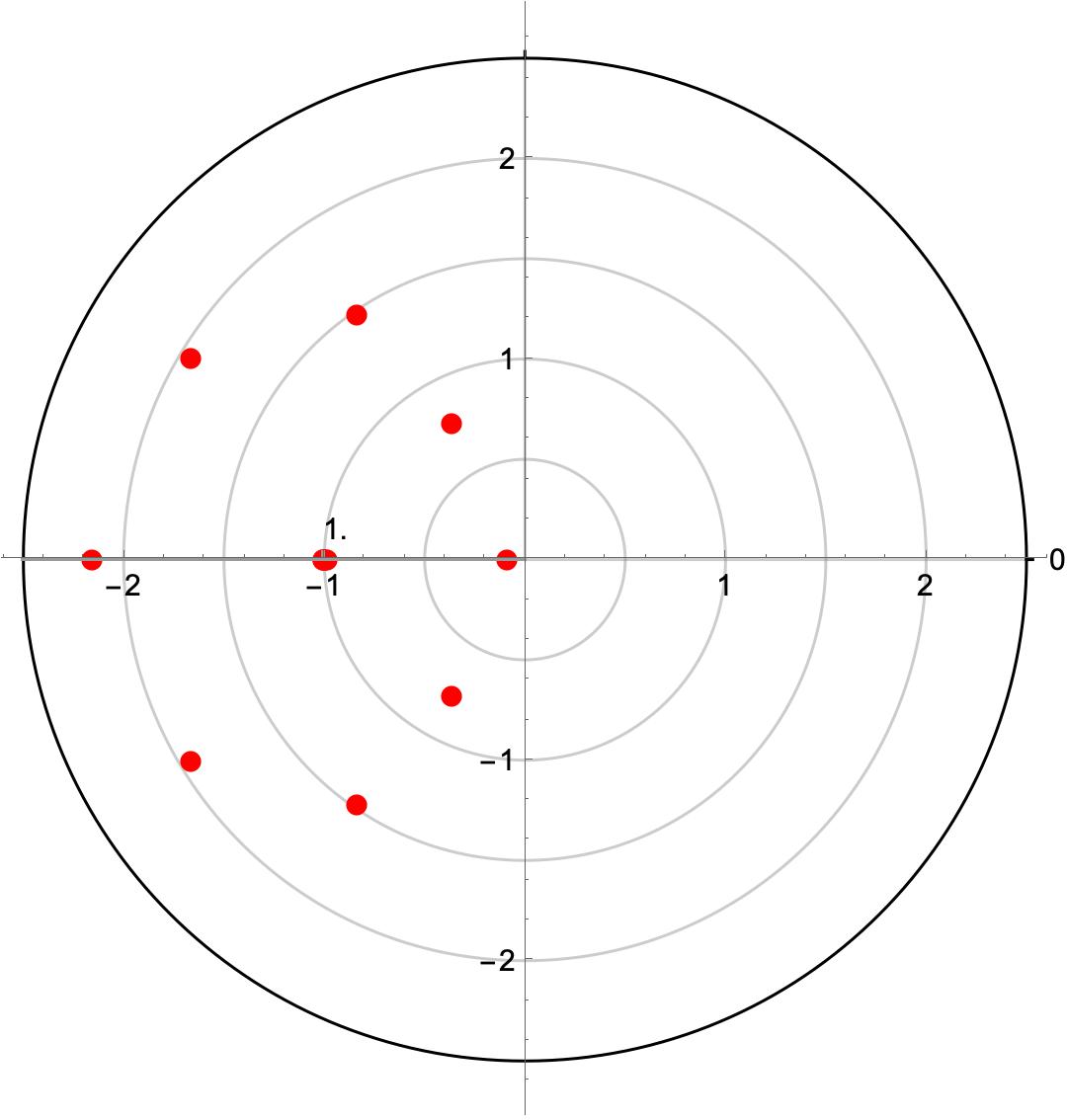}}~\scalebox{.34}{\includegraphics{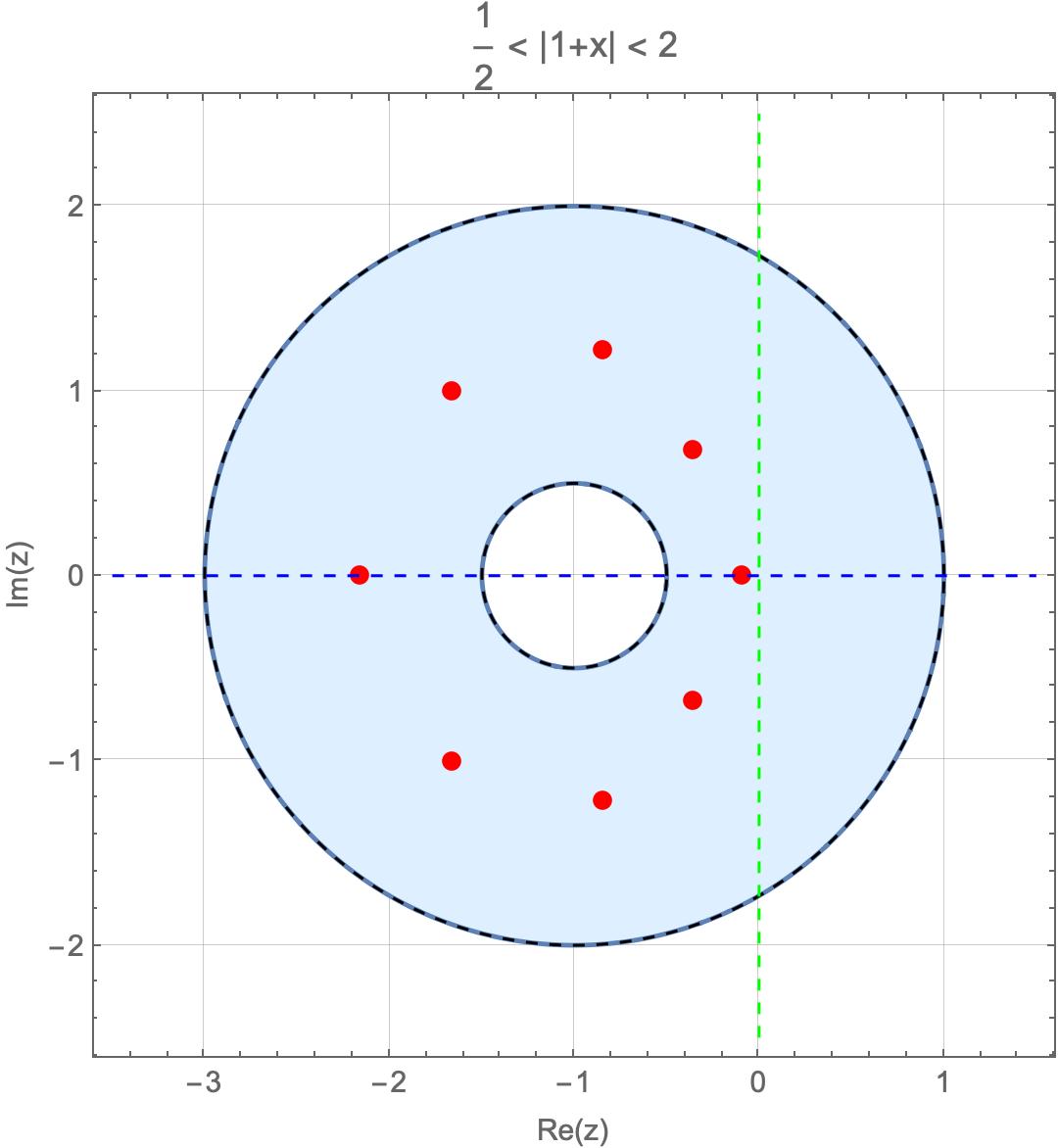}}}\\
	\qquad(a) Zeros of $I(\varGamma(\mathbb{Z}_{45}),x)$\qquad\qquad\qquad (b) Annular region $\tfrac{1}{2}<|1+x|<2$.
	\caption{Zeros of polynomial $I(\varGamma(\mathbb{Z}_{45}),x)$ in plane and zeros in annular region $ \tfrac{1}{2}<|1+x|<2$.}
	\label{zeros independent 45}\end{figure}

The independence polynomial of $\varGamma(\mathbb{Z}_{n})$ for $n\in \{p,p^{2},p^{3},pq,p^{2}q\}$ are unimodal and log-concave. These families provide additional classes of zero-divisor graphs whose independence polynomials are unimodal/log-concave, and contributes to the unimodal conjecture \cite{brown,micheal,alavi,levitejc,micheal,browncameron}.
With the computational experiments and  results presents in this article. We leave the following conjectures for the independence polynomials of zero divisor graph of rings.
\begin{conjecture}
	The independence polynomial of zero divisor graph of ring $\mathbb{Z}_{n}$ is unimodal.
\end{conjecture}
How about unimodality of  the independence polynomial of any zero divisor graph?

\section{Independence polynomial of zero divisor graph of order $pqr$}\label{section 4}
In this section, we give the independence polynomial $I(\varGamma(\mathbb{Z}_{n}),x)$ of graph $\varGamma(\mathbb{Z}_{n}),$ when $n=pqr$, where $p<q<r$ are primes. Then we find the bounds for the zeros of $I(\varGamma(\mathbb{Z}_{n}),x)$ in plane.
\begin{theorem}\label{in poly of zpqr}
	For the commutative ring $ \mathbb{Z}_{pqr}$ with primes $p<q<r$, the independence polynomial of $\varGamma(\mathbb{Z}_{pqr})$ is
	\begin{align*}
		I(\varGamma(\mathbb{Z}_{pqr}),x)=& (1+x)^{pq+qr+pr-2(p+q+r)+3} + (1+x)^{(p-1)(q+r-1)} - (1+x)^{(p-1)(q+r-2)}\\
		& + (1+x)^{(q-1)(p+r-1)} - (1+x)^{(q-1)(p+r-2)} + (1+x)^{(r-1)(p+q-1)}\\
		&- (1+x)^{(r-1)(p+q-2)}.
	\end{align*}
\end{theorem}
\begin{proof}
	Let $n=pqr$ with $p<q<r$ being prime numbers, and let $ G\cong \varGamma(\mathbb{Z}_{pqr})$ be its zero divisor graph.The set of vertices $V(G)$ can be partitioned based on the greatest common divisor (gcd) with $n$. For a divisor $d$ of $n$, let $V_d = \{v \in V : \gcd(v, n) = d\}$. The possible values for $d$ for any non-zero zero divisor are $p, q, r, pq, pr, $ and $qr$. The set of vertices is the disjoint union $V = V_p \cup V_q \cup V_r \cup V_{pq} \cup V_{pr} \cup V_{qr}$. The size of each partition is given by Euler's totient function $|V_d| = \varphi(\tfrac n d)$. Thus, we have
\begin{align*}
	|V_p| &= \varphi(qr) = (q-1)(r-1), |V_q| = \varphi(pr) = (p-1)(r-1), |V_{pq}|= \varphi(r) = r-1,\\
	|V_r| &= \varphi(pq) = (p-1)(q-1),	  |V_{pr}| = \varphi(q) = q-1, |V_{qr}| = \varphi(p) = p-1.
\end{align*}
Let $a \in V_{\rho_1}$ and $b \in V_{\rho_2}$ be two distinct vertices. Their product $ab$ is congruent to $0 \pmod n$ if and only if for each prime $s \in \{p,q,r\}$, $s \mid ab$. An element $v \in V_d$ is a multiple of $d$ and is not divisible by any prime factor of $\tfrac n d$. Thus, $s \mid ab$ if and only if $s \mid \rho_1$ or $s \mid \rho_2$. Therefore, $a$ and $b$ are adjacent if and only if $\operatorname{lcm}(\rho_1, \rho_2) = pqr$. Also, each of $V_{d}$ are independent set, since for any $a,b \in V_d$, $\operatorname{lcm}(d,d) = d \neq pqr$, where $d\in \{p, q, r, pq, pr, qr\}$ and lcm is their least common multiple. The adjacency relations among $V_{d}$'s are: $V_p$ is completely connected to $V_{qr}$, $V_q$ is completely connected to $V_{pr}$, $V_r$ is completely connected to $V_{pq}$. The sets $V_{pq}, V_{pr}, V_{qr}$ are pairwise completely connected, forming a complete tripartite subgraph, see Figure \ref{fig:zero-divisor-pqr-partition}.  
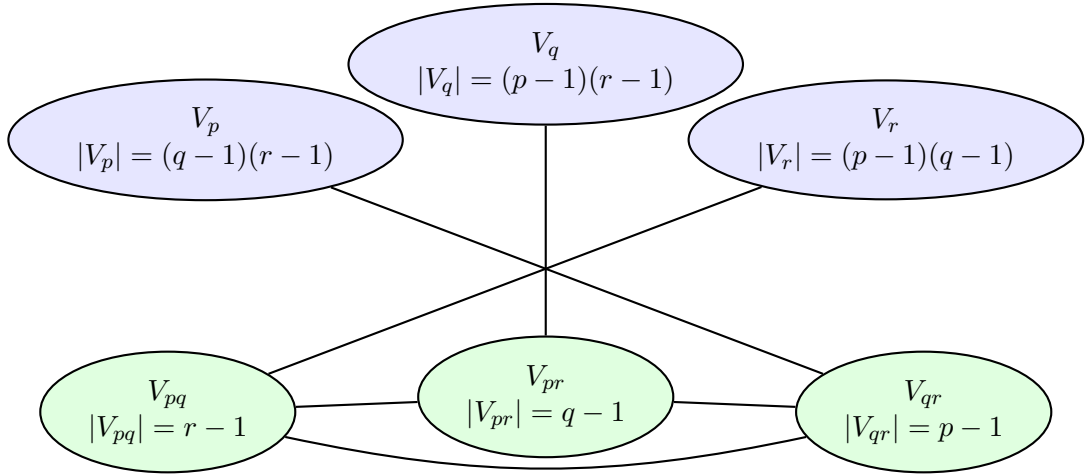
\begin{figure}[H]
	\centering
	\begin{tikzpicture}[
		part/.style={
			ellipse,
			draw,
			thick,
			minimum width=3.3cm,
			minimum height=1.25cm,
			align=center,
			font=\small
		},
		edge/.style={thick},
		complete/.style={thick}
		]
		
		\node[part, fill=blue!10] (Vp)  at (-4.5,1.4)
		{$V_p$\\ $|V_p|=(q-1)(r-1)$};
		
		\node[part, fill=blue!10] (Vq)  at (0,2.4)
		{$V_q$\\ $|V_q|=(p-1)(r-1)$};
		
		\node[part, fill=blue!10] (Vr)  at (4.5,1.4)
		{$V_r$\\ $|V_r|=(p-1)(q-1)$};
		
		\node[part, fill=green!12] (Vpq) at (-5,-2.2)
		{$V_{pq}$\\ $|V_{pq}|=r-1$};
		
		\node[part, fill=green!12] (Vpr) at (0,-2)
		{$V_{pr}$\\ $|V_{pr}|=q-1$};
		
		\node[part, fill=green!12] (Vqr) at (5,-2.2)
		{$V_{qr}$\\ $|V_{qr}|=p-1$};
		
		\draw[complete] (Vp) -- (Vqr);
		\draw[complete] (Vq) -- (Vpr);
		\draw[complete] (Vr) -- (Vpq);
		
		\draw[complete] (Vpq) -- (Vpr);
		\draw[complete] (Vpr) -- (Vqr);
		\draw[complete] (Vpq) to[bend right=12] (Vqr);
		
	\end{tikzpicture}
	\caption{Partition structure of the zero divisor graph $\varGamma(\mathbb{Z}_{pqr})$, where $p<q<r$ are primes.}
	\label{fig:zero-divisor-pqr-partition}
\end{figure} 
Next, we classify all independent sets in $G.$
Let $\mathcal{I}$ be the set of all independent sets of $\varGamma(\mathbb{Z}_{pqr})$. An independent set cannot contain vertices from two connected partitions. In particular, since $V_{pq}, V_{pr}, V_{qr}$ are pairwise connected, any independent set can have a non-empty intersection with at most one of them. This allows us to partition $\mathcal{I}$ into four disjoint families: 
\begin{align*}
	\mathcal{I}_0 &= \{I \in \mathcal{I} \mid I \subseteq V_p \cup V_q \cup V_r\},\qquad \mathcal{I}_1 = \{I \in \mathcal{I} \mid I \cap V_{pq} \neq \emptyset\},\\
 	\mathcal{I}_2 &= \{I \in \mathcal{I} \mid I \cap V_{pr} \neq \emptyset\},\qquad \mathcal{I}_3 = \{I \in \mathcal{I} \mid I \cap V_{qr} \neq \emptyset\}.
\end{align*}
The independence polynomial $I(x) = \sum_{I \in \mathcal{I}} x^{|I|}$ is the sum of the polynomials for each family. Let $I_k(x)$ be the polynomial for family $\mathcal{I}_k$. For the family $\mathcal{I}_0$, an independent set $I \in \mathcal{I}_0$ is a subset of $V_p \cup V_q \cup V_r$. Since there are no edges within or between these partitions, any subset of $V_p \cup V_q \cup V_r$ is an independent set. Their cardinal number is $|V_p|+|V_q|+|V_r|$, and the corresponding polynomial is
$$I_0(x) = (1+x)^{|V_p|+|V_q|+|V_r|} = (1+x)^{(q-1)(r-1)+(p-1)(r-1)+(p-1)(q-1)}.$$
For the family $\mathcal{I}_1$, we have $I \in \mathcal{I}_1$, and $I \cap V_{pq} \neq \emptyset$. This implies that $I \cap V_{pr} = \emptyset$ and $I \cap V_{qr} = \emptyset$. Due to the edges between $V_{pq}$ and $V_r$, we must have $I \cap V_r = \emptyset$. Thus, $I$ must be a subset of $V_p \cup V_q \cup V_{pq}$. This union of sets induces a subgraph with no edges, so any of its subsets is an independent set. The condition for being in $\mathcal{I}_1$ is that the subset must have a non-empty intersection with $V_{pq}$.
The generating function for all subsets of $V_p \cup V_q \cup V_{pq}$ is $(1+x)^{|V_p|}(1+x)^{|V_q|}(1+x)^{|V_r|},$ which is equivalent to
$(1+x)^{|V_p|+|V_q|+|V_{pq}|}$.
The generating function for subsets that do not intersect $V_{pq}$ (that is, subsets of $V_p \cup V_q$) is $(1+x)^{|V_p|+|V_q|}$.
So, we obtain $$I_1(x) = (1+x)^{|V_p|+|V_q|+|V_{pq}|} - (1+x)^{|V_p|+|V_q|}.$$
Now, with the values $$|V_p|+|V_q|+|V_{pq}| = (r-1)(q-1+p-1+1) = (r-1)(p+q-1),$$
and  $|V_p|+|V_q| = (q-1)(r-1)+(p-1)(r-1) = (r-1)(p+q-2)$. Thus,  we have
$$I_1(x) = (1+x)^{(r-1)(p+q-1)} - (1+x)^{(r-1)(p+q-2)}.$$
Again for the family $\mathcal{I}_2$, we have  $I \in \mathcal{I}_2$, and $I \cap V_{pr} \neq \emptyset$. This implies that $I \cap V_q = \emptyset$. So, $I$ is a subset of $V_p \cup V_r \cup V_{pr}$ with $I \cap V_{pr} \neq \emptyset$. With the similar idea as above we have
$$|V_p|+|V_r|+|V_{pr}| = (q-1)(r-1)+(p-1)(q-1)+(q-1) = (q-1)(p+r-1),$$
and  $|V_p|+|V_r| = (q-1)(r-1)+(p-1)(q-1) = (q-1)(p+r-2)$. Thus, with these values, we have
$$I_2(x) = (1+x)^{(q-1)(p+r-1)} - (1+x)^{(q-1)(p+r-2)}.$$
Finally for the family $\mathcal{I}_3$, we have $I \in \mathcal{I}_3$, $I \cap V_{qr} \neq \emptyset$. This implies that $I \cap V_p = \emptyset$. So, $I$ is a subset of $V_q \cup V_r \cup V_{qr}$ with $I \cap V_{qr} \neq \emptyset$. Thus, the size of sets are 
$|V_q|+|V_r|+|V_{qr}| = (p-1)(r-1)+(p-1)(q-1)+(p-1) = (p-1)(q+r-1)$, and  $|V_q|+|V_r| = (p-1)(r-1)+(p-1)(q-1) = (p-1)(q+r-2)$. The generating function is
$$I_3(x) = (1+x)^{(p-1)(q+r-1)} - (1+x)^{(p-1)(q+r-2)}.$$
The total independence polynomial $I(\varGamma(\mathbb{Z}_{pqr}),x)$ of $\varGamma(\mathbb{Z}_{pqr})$ is $I(\varGamma(\mathbb{Z}_{pqr}),x) = I_0(x) + I_1(x) + I_2(x) + I_3(x)$. Substituting the expressions and sizes gives the final expression as in statement.
\end{proof}

Next, we discuss the zeros of the polynomial $I(\varGamma(\mathbb{Z}_{n}),x)$ in plane and illustrate result with an example.
\begin{theorem}\label{zeros pqr}
	Let $I(\varGamma(\mathbb{Z}_{pqr}),x)$ be the independence polynomial of $\varGamma(\mathbb{Z}_{pqr})$, where $p<q<r$ are primes. Then $x=-1$ is a zero of $I(\varGamma(\mathbb{Z}_{pqr}),x)$ with multiplicity $(p-1)(q+r-2)$, and every other zero $x$ of $I(\varGamma(\mathbb{Z}_{pqr}),x)$ satisfies
	$
	\tfrac12<|1+x|<2.
	$
\end{theorem}

\begin{proof}
	We use compact notations with 
	$a=p-1,b=q-1,$ and $ c=r-1.$ So, it is clear that  $1\leq a<b<c$. In the independence polynomial of $\varGamma(\mathbb{Z}_{pqr})$, we let $y=1+x$, and obtain
	\begin{align*}
		I(y-1)
		=&\,y^{ab+bc+ca}+y^{ab+ac+a}-y^{ab+ac}+y^{ab+bc+b}-y^{ab+bc}+y^{ac+bc+c}-y^{ac+bc}.
	\end{align*}
	Since $a<b<c$, the smallest exponent appearing in this expression is $ab+ac=a(b+c)$, and its coefficient is $-1$. Hence, $y=0$ is a zero of $I(y-1)$ with multiplicity $a(b+c)$, and thereby  $x=-1$ is a zero of $I(x)$ with multiplicity
	$
	a(b+c)=(p-1)(q+r-2).
	$We can put $I(y-1)$ as $I(y-1)=y^{ab+ac}H(y),$ 
	where
	$$
	H(y)
	=
	-1+y^a+y^{bc}+y^{c(b-a)+b}-y^{c(b-a)}+y^{b(c-a)+c}-y^{b(c-a)}.
	$$
	 We first prove that $H(y)$ has no zeros in $|y|\leq \tfrac12$. On the circle $|y|=\tfrac12$, we write
	$
	H(y)=-1+S(y),
	$
	where
	$$
	S(y)=y^a+y^{bc}+y^{c(b-a)+b}-y^{c(b-a)}+y^{b(c-a)+c}-y^{b(c-a)}.
	$$
	Since $a\geq 1$, $bc\geq 8$, $c(b-a)\geq 4$, $c(b-a)+b\geq 6$, $b(c-a)\geq 6$, and $b(c-a)+c\geq 10$, we have
	$$
	|S(y)|
	\leq
	2^{-a}+2^{-bc}+2^{-(c(b-a)+b)}+2^{-c(b-a)}
	+2^{-(b(c-a)+c)}+2^{-b(c-a)}.
	$$
	Thus,
	$$
	|S(y)|
	\leq
	\frac12+\frac{1}{256}+\frac{1}{64}+\frac{1}{16}
	+\frac{1}{1024}+\frac{1}{64}
	<1.
	$$
	Hence, on $|y|=\tfrac12$, we have
	$
	|S(y)|<|-1|=1.
	$
	By Rouché's theorem, $H(y)$ and the constant polynomial $-1$ have the same number of zeros in $|y|<\tfrac12$. Therefore, $H(y)$ has no zeros in $|y|<\tfrac12$. The strict inequality also shows that $H(y)\neq 0$ on $|y|=\tfrac12$. Hence, every zero of $H(y)$ satisfies
	$
	|y|>\tfrac12.
	$ Next, we prove that every zero of $H(y)$ satisfies $|y|<2$. Let
	$
	H(y)=\sum_{j=0}^{d}c_jy^j,
	$
	where $d=\deg H$ and $c_d\neq 0$. We claim that
	$|c_j|\leq |c_d|$, for all $0\leq j<d.$
	Clearly, the only possible coincidence between two positive exponents in
	$
	a, bc,c(b-a)+b,$ and $ b(c-a)+c
	$
	is
	$
	bc=b(c-a)+c,
	$
	which is equivalent to $c=ab$. If this occurs, it occurs at the largest exponent and gives the leading coefficient $c_d=2$. Otherwise the leading coefficient has modulus $1$. The two negative exponents $c(b-a)$ and $b(c-a)$ are distinct, and any coincidence between a positive and a negative exponent only causes cancellation. Therefore, every non-leading coefficient has modulus at most $1$, while the leading coefficient has modulus either $1$ or $2$. Now, let 
	$
	H(y)=c_dy^d+R(y),
	$
	where
	$
	R(y)=\sum_{j=0}^{d-1}c_jy^j.
	$
	On the circle $|y|=2$, we have
	$$
	|R(y)|
	\leq
	\sum_{j=0}^{d-1}|c_j|2^j
	\leq
	|c_d|\sum_{j=0}^{d-1}2^j
	=
	|c_d|(2^d-1),
	$$
	and hence 
	$
	|R(y)|<|c_d|2^d=|c_dy^d|.
	$
	By Rouché's theorem, $H(y)$ and $c_dy^d$ have the same number of zeros in $|y|<2$, counted with multiplicity. Since $H$ has degree $d$, all zeros of $H$ lie in $|y|<2$. The strict inequality also excludes zeros on $|y|=2$.
	 Combining the two facts, every zero $y$ of $H(y)$ satisfies
	$
	\tfrac12<|y|<2.
	$
	Thus, with $y=1+x$, every zero $x\neq -1$ of $I(\varGamma(\mathbb{Z}_{pqr}),x)$ satisfies
	$
	\tfrac12<|1+x|<2.
	$ 
\end{proof}

The independence polynomial of $\varGamma(\mathbb{Z}_{30})$ is 
\begin{align*}
	I(\varGamma(\mathbb{Z}_{30}))=&(x+1)^{16}+(x+1)^{14}-(x+1)^{10}+(x+1)^7-(x+1)^6\\
	=&(x+1)^6 (x^{10}+10 x^9+46 x^8+128 x^7+238 x^6+308 x^5+279 x^4+172 x^3\\
	&+67 x^2+15 x+1).
\end{align*}
The zeros of $I(\varGamma(\mathbb{Z}_{30})) $ are 
\begin{align*}
	&(-1)^{6},-2.05728,-0.104067,-1.68807\pm 0.681782 i,-1.29235\pm 1.12643 i,\\
	&-0.588973\pm 1.08943 i, -0.349935\pm  0.436093i.
\end{align*}
The zero set of $ I(\varGamma(\mathbb{Z}_{30}),x)$ is plotted in (a) Figure \ref{zeros independent 30}, and its zeros $x$ (other than $x=-1$) lying in  the annular region $\tfrac{1}{2}<|1+x|<2 $ is shown in (b) Figure \ref{zeros independent 30}.
\begin{figure}[H]
	\centering{\scalebox{.3}{\includegraphics{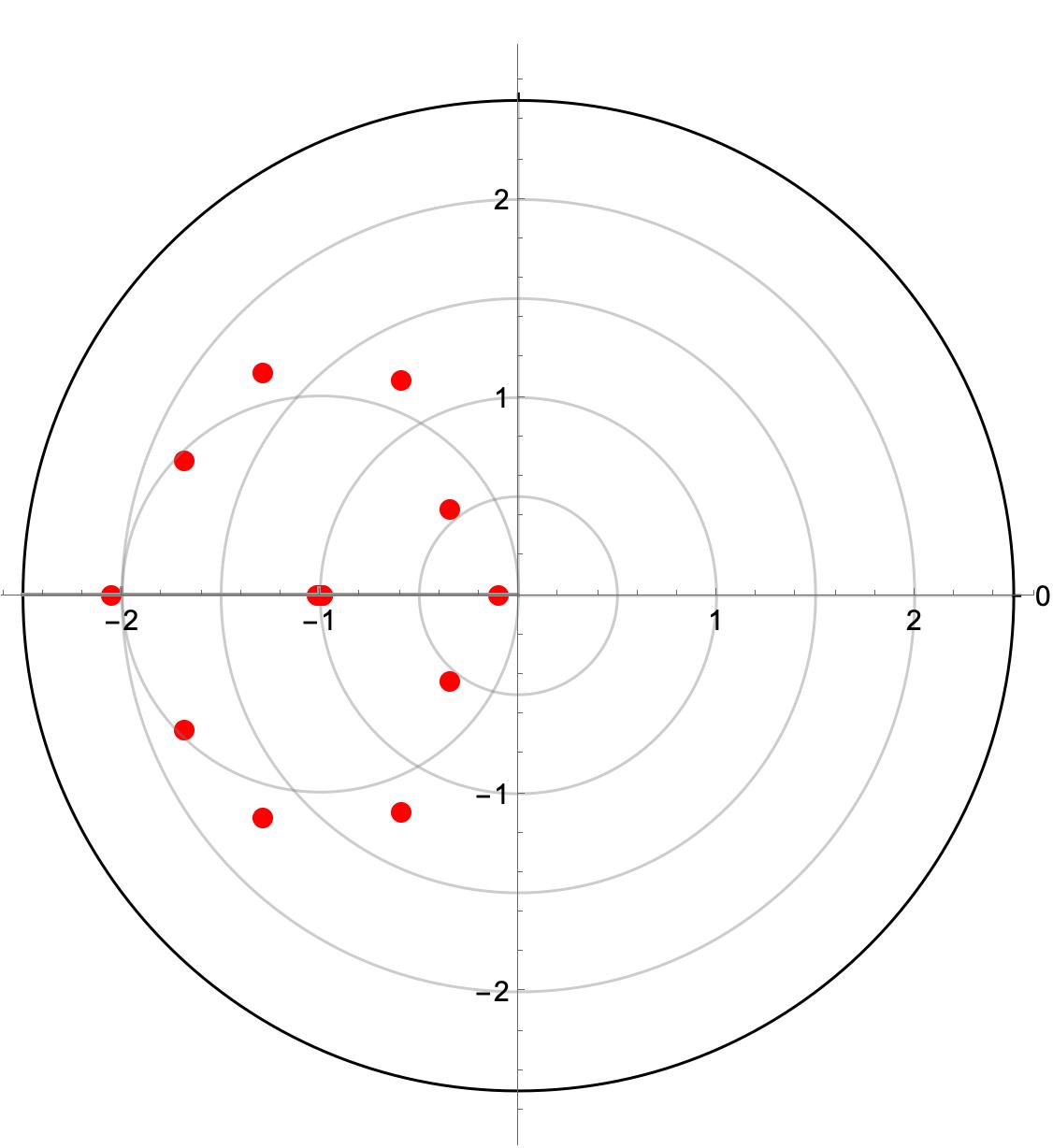}}\quad  \scalebox{.31}{\includegraphics{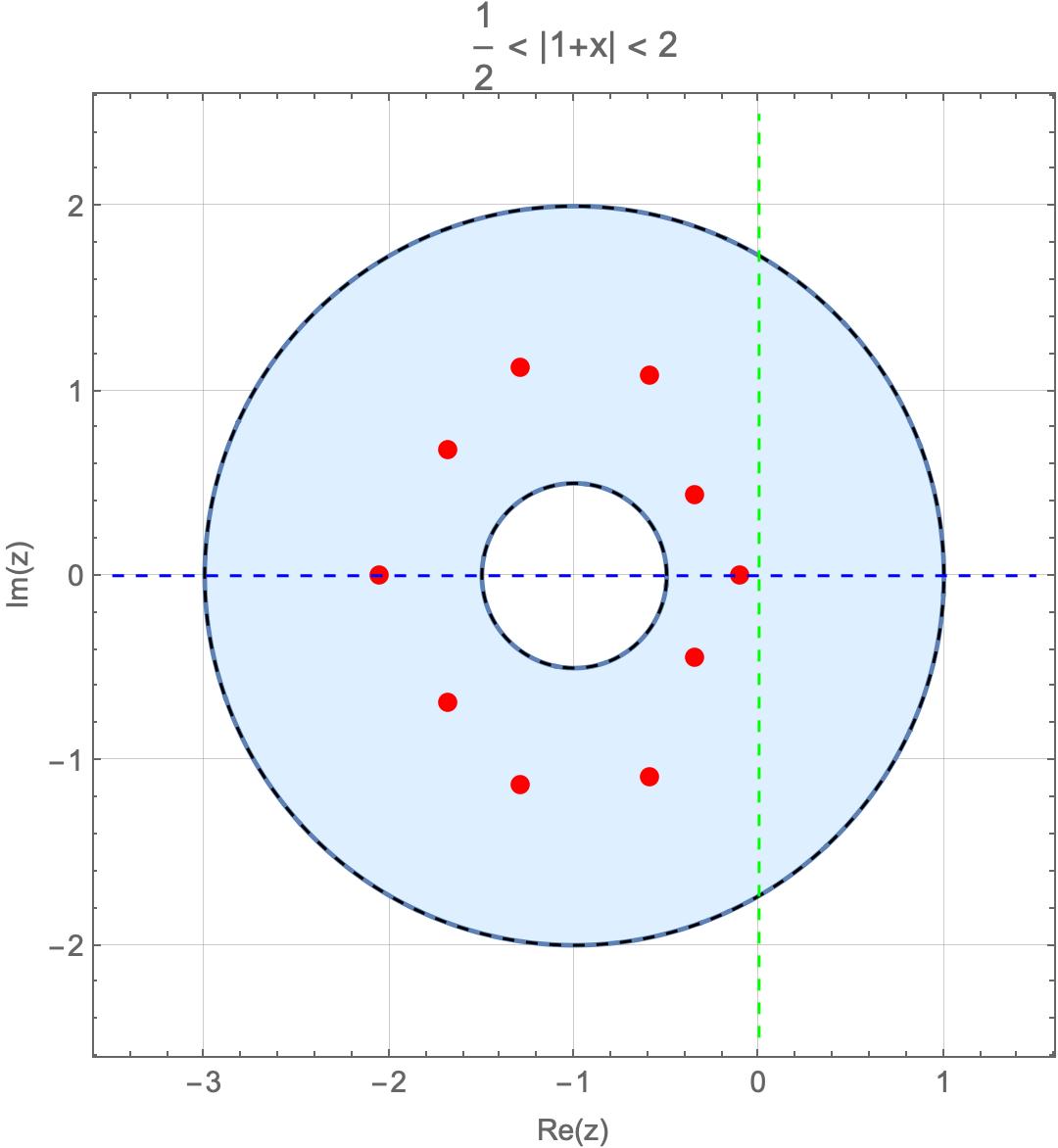}}}\\
	\quad (a) Zeros of $I(\varGamma(\mathbb{Z}_{30}),x)$ \qquad \quad \qquad (b) Annular $\tfrac{1}{2}<|1+x|<2$
	\caption{Zeros of  $I(\varGamma(\mathbb{Z}_{30}),x)$, and its complex zero $z$ satisfy $\tfrac{1}{2}<|1+x|<2. $}
	\label{zeros independent 30}
\end{figure}

Eneström-Kakeya theorem (see, \cite{barbeau}),  states that:  Given a real polynomial $e(x)=\sum_{i=0}^{n} e_{i}x^{i}$ with $e_{i}\geq 0$,  the zeros of $e(x)$ lie in the annulus region: $r\leq |x|\leq R,$
where 
$$r=\min \left \{\left|\frac{e_{i}}{e_{i+1}}\right| : 0\leq i\leq n-1  \right \} ~\text{and}~R=\max \left  \{\left|\frac{e_{i}}{e_{i+1}}\right| : 0\leq i\leq n-1  \right \}.$$ Next, we apply Eneström-Kakeya theorem to polynomial $I(\varGamma(\mathbb{Z}_{pqr}),x) $ and compare the region with Theorem \ref{zeros pqr}. We have the following proposition.
\begin{proposition}
	The maximum exponent in the polynomial $I(\varGamma(\mathbb{Z}_{pqr}),x)$ is $(r-1)(p+q-1)$ if $p=2$, and for $p \ge 3$ it is $(r-1)(p+q-1)$, if $r-1 > (p-1)(q-1)$ and $pq+qr+pr-2(p+q+r)+3$, if $r-1 < (p-1)(q-1)$. If $r-1=(p-1)(q-1)$, these two expressions are equal and maximal. The minimal exponent (degree) of $I(\varGamma(\mathbb{Z}_{pqr}),x)$ is $(p-1)(q+r-2)$.
\end{proposition}
\begin{proof} Let $I(\varGamma(\mathbb{Z}_{pqr}),x) $ be the independence polynomial with  primes be $p<q<r$, and let $x=p-1, y=q-1, z=r-1$. Since $p,q,r$ are primes, $p \ge 2, q \ge 3, r \ge 5$, which implies $x \ge 1, y \ge 2, z \ge 4$. The condition $p<q<r$ implies $x<y<z$. In order to simplify calculations, we consider the following assignments 
	\begin{align*}
	 A &= pq + qr + pr - 2(p + q + r) + 3 = (p-1)(q-1) + (q-1)(r-1) + (r-1)(p-1) \\
	 &= xy+yz+zx,\\
	 A_1 &= (p-1)(q+r-1) = x(y+1+z+1-1) = x(y+z+1) = xy+xz+x,\\
	 A_2 &= (p-1)(q+r-2) = x(y+1+z+1-2) = x(y+z) = xy+xz,\\
	 C_1 &= (q-1)(p+r-1) = y(x+1+z+1-1) = y(x+z+1) = xy+yz+y,\\
	 C_2 &= (q-1)(p+r-2) = y(x+1+z+1-2) = y(x+z) = xy+yz,\\
	 D_1 &= (r-1)(p+q-1) = z(x+1+y+1-1) = z(x+y+1) = xz+yz+z,\\
	 D_2 &= (r-1)(p+q-2) = z(x+1+y+1-2) = z(x+y) = xz+yz.
	\end{align*}
	Next, we we establish clear inequalities within above pairs of terms.  Consider $A_1 - A_2 = (xy+xz+x) - (xy+xz) = x = p-1 > 0, $ it implies that $ A_1 > A_2$. Similarly,  $C_1 - C_2 = (xy+yz+y) - (xy+yz) = y = q-1 > 0 $ implies that $C_1 > C_2$, and  $D_1 - D_2 = (xz+yz+z) - (xz+yz) = z = r-1 > 0 $ implies that $ D_1 > D_2$. Now, we look for minimum among  $\{A, A_2, C_2, D_2\}$.  Consider $C_2 - A_2 = (xy+yz) - (xy+xz) = yz-xz = z(y-x)$, then we get  $C_2 > A_2$, as  $z>0$ and $y>x$. Also, $D_2 - C_2 = (xz+yz) - (xy+yz) = xz-xy = x(z-y)$ implies that $D_2 > C_2$, since $x>0$ and $z>y$. Thus, $A_2 < C_2 < D_2$. Again, $A - A_2 = (xy+yz+zx) - (xy+xz) = yz$ gives us $A > A_2$, as $y,z>0$. Thus, from these comparisons, $A_2=(p-1)(q+r-2)$ is the smallest among $\{A, A_2, C_2, D_2\}$. Now, we search for the maximum exponent among $\{A, A_1, C_1, D_1\}$. With $C_1 - A_1 = (xy+yz+y) - (xy+xz+x) = yz-xz+y-x = z(y-x)+(y-x) = (z+1)(y-x)$ we have $C_1 > A_1$, as$z>0, y >x$. Again from
 	$D_1 - C_1 = (xz+yz+z) - (xy+yz+y) = xz-xy+z-y = x(z-y)+(z-y) = (x+1)(z-y)$ we have $D_1 > C_1$, since $x>0, z>y$. So, we get ordering $A_1 < C_1 < D_1$. Now compare $A$ with $C_1$, we have
	$$A - C_1 = (xy+yz+zx) - (xy+yz+y) = zx-y = (r-1)(p-1)-(q-1).$$
	If $p=2$, $A-C_1 = (r-1)-(q-1) = r-q > 0$.  If $p \ge 3$, then $p,q,r$ are odd primes, so $q \ge p+2$ and $r \ge q+2$, and 
	$$A-C_1 = (p-1)(r-1)-(q-1) \ge (p-1)(q+1)-(q-1) =  pq+p-2q = q(p-2)+p.$$
	As $p \ge 3$, $p-2 \ge 1$, so $q(p-2)+p > 0$, and from above we get $A > C_1$, which is valid for all primes $p,q,r$. Thus, the maximum exponent is either $A$ or $D_1$. Now, we compare them, and find the maximum exponent.  With $D_1 - A = (xz+yz+z) - (xy+yz+zx) = z-xy = (r-1)-(p-1)(q-1)$, we see that  the sign of this difference is not constant, and depends on the specific values of $p,q,r$. If $p=2$, then $D_1 - A = (r-1) - (2-1)(q-1) = r-1-q+1 = r-q > 0$. So, in this case, $D_1$ is always greater than $A$.  If $p \ge 3$, then the sign of $D_1-A$ depends on the sign of $(r-1)-(p-1)(q-1)$. If $r\geq 11$, then  $(r-1)-(p-1)(q-1)> 0$, and it follows that $D_1 > A$. Therefore, the maximum element is not represented by a single expression for all primes $p<q<r$. The maximum is $D_1$ when $(r-1) > (p-1)(q-1)$ and $A$ when $(r-1) < (p-1)(q-1)$.
\end{proof}
The minimum $r$ among the coefficients of $I(\varGamma(\mathbb{Z}_{pqr}),x)$ is $r=p q + q r + p r - (p + q + r)$, since the coefficient of linear term in $ I(\varGamma(\mathbb{Z}_{pqr}),x)$ is 
\begin{align*}
	-&((p-1) (q+r-2))+(p-1) (q+r-1)-(r-1) (p+q-2)-(q-1) (p+r-2)\\
	&+(r-1) (p+q-1)+(q-1) (p+r-1)-2 (p+q+r)+p q+p r+q r+3\\
	&= p q + q r + p r - (p + q + r).
\end{align*}
By applying the Eneström-Kakeya theorem to $I(\varGamma(\mathbb{Z}_{pqr}),x), $ we have the following result.
\begin{theorem}\rm\label{kakeya}
	The zeros of the polynomial $I(x)$ lie in the annulus region 
	$$\frac{1}{p q + q r + p r - (p + q + r)} \leq |Z| \leq R,$$ where $R$ is given as:
	\begin{enumerate}
		\item If $p=2$, then $R=(r-1)(p+q-1).$
		\item If $p \ge 3$, then $R=(r-1)(p+q-1)$ provided $r-1 > (p-1)(q-1)$.
		\item  If $p \ge 3$, then $R= pq+qr+pr-2(p+q+r)+3$ provided $r-1 < (p-1)(q-1)$.
	\end{enumerate}
\end{theorem}
Let $P(x)=I(\varGamma(\mathbb{Z}_{pqr}),x)$ be the independence polynomial of zero divisor graph $\varGamma(\mathbb{Z}_{pqr}).$ Now, with $p=2,q=3$ and $5=5$, by Eneström-Kakeya theorem, the zeros of $P(x) $ lie in the annular region $\tfrac{1}{21}\leq |x|\leq 16$, which is much bigger region as compared to one given in Theorem \ref{zeros pqr}. Thus, Theorem \ref{zeros pqr} gives better bounds as compared to Theorem \ref{kakeya}. \vskip 2mm

As the coefficients of $I(\varGamma(\mathbb{Z}_{pqr}),x)$ are positive, and it is known that a positive log-concave sequence is unimodal \cite{stanley}. So, if we prove log-concavity of $I(\varGamma(\mathbb{Z}_{pqr}),x),$ then the unimodal proper will follow. Also, each   each binomial term in $ I(\varGamma(\mathbb{Z}_{pqr}),x)$ is log-concave due to symmetric behaviour of binomial coefficients. However sum of two log-concave polynomial is not necessarily log-concave. But due to symmetric behaviour of binomial terms in $I(\varGamma(\mathbb{Z}_{pqr}),x),$ it looks like polynomial is log-concave and thereby unimodal. We leave it as a problem. 
\begin{problem}
	Let $I(\varGamma(\mathbb{Z}_{pqr}),x)$ be the independence polynomial of $\varGamma(\mathbb{Z}_{pqr})$, with primes $p<q<r$. Then prove that $I(\varGamma(\mathbb{Z}_{pqr}),x)$ is a unimodal and log-concave polynomial.
\end{problem}
\section{Conclusion}
This article examines the independence polynomial $I(\varGamma(\mathbb{Z}_{n},x)$ of the zero divisor graph $\varGamma(\mathbb{Z}_{n})$ of commutative rings $\mathbb{Z}_{n}$ for $n \in \{p,p^{2},p^{3},pq,p^{2}q,pqr\}$, where $p<q<r$ are prime integers. The unimodal and log-concave properties of the previously stated graphs are analyzed. The positions of the zeros $I(\varGamma(\mathbb{Z}_{n}),x)$ in the plane are delineated, supplemented by pictures obtained from numerical computations. The independence polynomial's application can be broadened to further values of $n$ in the graph $\varGamma(\mathbb{Z}_{n})$, and its components may be explored in subsequent research.

\section*{Declarations}
\noindent \textbf{Data Availability:}	There is no data associated with this article.

\noindent \textbf{Funding:} The authors did not receive support from any organization for the submitted work.

\noindent \textbf{Conflict of interest:} The authors have no competing interests to declare that are relevant to the content of this article. 

\noindent\textbf{Acknowledgement:} 
The authors acknowledge the use of Wolfram Mathematica (Version 13.1) for generating the figures presented in this paper. Additionally, online TeX tools were employed for drawing TikZ figures. QuillBot was used for language polishing and grammar checking of this manuscript.

\noindent\textbf{Note:} For any comments and suggestions regarding this article, please feel free to contact at \href{mailto:bilalahmadrr@gmail.com}{bilalahmadrr@gmail.com}.


\begin{thebibliography}{0}
	\bibitem{alavi} Y. Alavi, P. J. Malde, A. J. Schwenk and P. Erd\"os, The vertex independence sequence of a graph is not constrained, \emph{Congr. Numer.} \textbf{58} (1987) 15--23.
	\bibitem{alikhanizd} S. Alikhani and F. Aghaei, Domination polynomial and total domination polynomial of zero divisor graphs of commutative rings, (2019), \url{https://arxiv.org/abs/2404.13539}.
	\bibitem{anderson2019} {D. F. Anderson and D. Weber, The zero divisor graph of a commutative ring without identity, \em Int. Elect. J. Algebra} {\bf 223} (2018) 176--202.
	\bibitem{Andersonbook} D. F. Anderson, T. Asir, A. Badawi and T. Tamizh Chelvam,
	\textit{Graphs from rings}, Springer, 2021.
	\bibitem{al} { D. F. Anderson and P. S. Livingston, The zero divisor graph of a commutative ring, \em J. Algebra} {\bf 217} (1999) 434--447.
	
	\bibitem{arunkumara} G. Arunkumara, P. J. Cameron, T. Kavaskar and T. T. Chelvam, Induced subgraphs of zero divisor graphs, \emph{Discrete Math.} \textbf{346} (2023) 113580. 
	
	\bibitem{barbeau} E. J. Barbeau, \textit{Polynomials}, Springer-Verlag, New York,1989.
	
	\bibitem{ib}  {I. Beck, Coloring of a commutative rings, \em J. Algebra} {\bf 116} (1988) 208--226.
	
	\bibitem{brown} J. I. Brown, K. Dilcher and R.J. Nowakowski, Roots of independence polynomials of well covered graphs, \emph{J. Algebraic Combin.} \textbf{11} (2000) 197--210.
	\bibitem{browncameron} J. I. Brown and P. J. Cameron, On the unimodality of independence polynomials of very well-covered graphs, \textit{Dicrete Math.} {\bf 341}(4) (2011) 1138--1143.
	
	\bibitem{asir2} T. T. Chelvam, T. Asir and K. Selvakumar, On dominations in graphs from commutative rings: A survey, Springer Proceedings in Mathematics \& Statistics-Algebra and its Applications ICAA, Vol 174 (2016) 363-380.
	\bibitem{ChudnovskySeymour2007}
	M.~Chudnovsky and P.~Seymour, The roots of the independence polynomial of a clawfree graph, \emph{J.\ Combin.\ Theory Ser.\ B} \textbf{97} (2007), 350--357.
	\bibitem{Hastad99}
	J.~H{\aa}stad,
	\newblock Clique is hard to approximate within $n^{1-\varepsilon}$,
	\newblock \emph{Acta Mathematica} \textbf{182} (1999), 105--142.
	\bibitem{goddard} W. Goddard and M. A. Henning, Independent domination in graphs: a survey and recent results, \emph{Discrete Math.} \textbf{313} (2013) 839--854.
	\bibitem{garey1979}
	M. R. Garey and D. S. Johnson, \textit{Computers and Intractability: A Guide to the Theory of NP-Completeness}, W. H. Freeman and Company, New York,  1979.
	\bibitem{gross} J. L. Gross,T. Mansour, T. W. Tucker and D. G. L. Wang, Log-concavity of combinations of sequences and applications to genus distributions, \textit{SIAM J. Discrete Math.} \textbf{29}(2) (2015) 1002--1029.
	
	\bibitem{gutman} I. Gutman and F. Harary, Generalizations of the matching polynomials, \emph{Utilitas Math.} \textbf{24} (1983) 97--106.
	\bibitem{gursoy} N. K. G\"ursoy, A. \"Ulker and A. G\"ursoy, Independent domination polynomial of zero divisor graphs of commutative rings, \emph{Soft Comp.}  \textbf{26}(15) (2022) 6989--6997.
	
	\bibitem{haynes} T. W. Haynes, S. T. Hedetniemi and P. J. Slater, \emph{Fundamentals of Domination in Graphs}, Marcel Dekker, New York, 1998.
	\bibitem{heilmann} O. J. Heilmann and E. H. Lieb, Theory of monomer-dimer systems, \textit{Commun. Math. Phy.} \textbf{25}(3) (1972) 190--232.
	\bibitem{huh} J. Huh, Milnor numbers of projective hypersurfaces and the chromatic polynomial of a graph, \textit{J. Amer. Math. Soc.} \textbf{25} (2012) 907--927.
	\bibitem{huh1} J. Huh, $h$-Vectors of matroids and logarithmic concavity, \textit{Adv. Math.} \textbf{270} (2015) 49--59.
	
	\bibitem{karp1972} R. M. Karp, Reducibility among combinatorial problems. In R. E. Miller \& J. W. Thatcher (Eds.), \textit{Complexity of Computer Computations} (pp. 85–103). Plenum Press 1972.
	\bibitem{levitsurvey} V. E. Levit and E. Mandrescu, The independence polynomial of a graph-a survey, in: \emph{Proc. Int. Con. Algeb. Inf. Aristotle Univ. Thessaloniki, Greece} (2005) 233--254.
	\bibitem{levitdiscrete} V. E. Levit and E. Mandrescu, On the independence polynomial of the corona of graphs, \emph{Discrete Appl. Math.} \textbf{203} (2016) 85--93.
	\bibitem{levitejc} V. E. Levit and E. Mandrescu, Independence polynomials of well-covered graphs: generic counterexamples for the unimodality conjecture, \textit{Europ. J. Combin.} \textbf{27} (2006)  931--939.
	\bibitem{micheal} T. S. Michael Tand W. N. Traves, Independence sequences of well-covered graphs: non-unimodality and the roller-coaster conjecture, \textit{Graphs Combin.} \textbf{19} (2003) 403--411.
	
	\bibitem{bilaldml} B. A. Rather, Independent domination polynomial for the cozero divisor graph of the ring of integers modulo $n$, \emph{Discrete Math. Lett.} \textbf{13} (2024) 36--43.
	
	\bibitem{bilaltcs} B. A. Rather, Complex zeros and log-concavity in independent domination polynomials of zero divisor graphs of commutative rings,  \emph{Theoret. Comput. Sci.} \textbf{1058} (2025) 115594.
	
	\bibitem{bilalsc}  B. A. Rather, Complex Zeros of Independent Domination Polynomials of Zero Divisor Graphs, \emph{Soft Comput.} \textbf{30} (2026) 3009--3022. 
	
	\bibitem{bilaljcmcc} B. A. Rather, On domination polynomials of some graphs, \textit{J. Combin. Math. Combin. Comp.} \textbf{126} (2025) 279--289.

	\bibitem{bilaljac} B. A. Rather, J. Wang and F. Balarado, Independent domination polynomial of binary sequence graphs, \emph{J. Algebraic Combin.} \textbf{63} (2026), Art. No. 57. 

	\bibitem{stanley} R. P. Stanley, Log-concave and unimodal sequences in algebra, combinatorics and geometry, Graph Theory and Its Applications East and West: Proceedings of the First China‐USA International Graph Theory Conference, Annals New York Academy Sci. \textbf{567}(1) (1989) 500--535. 
	\bibitem{Weitz} D.~Weitz, Counting independent sets up to the tree threshold, STOC 06: Proceedings of the thirty-eighth annual ACM symposium on Theory of Computing, (2006)  140--149.
	
	
	
\end{thebibliography}
\end{document}